\documentclass{article} %arXiv
\usepackage[a4paper,margin=1in]{geometry}

\usepackage{amssymb,amsthm,amsmath}

\newtheorem{theorem}{Theorem}[section]
\newtheorem{lemma}[theorem]{Lemma}
\newtheorem{corollary}[theorem]{Corollary}

\newtheorem{fact}[theorem]{Fact}

\theoremstyle{definition}
\newtheorem{definition}[theorem]{Definition}
\newtheorem{example}[theorem]{Example}

\theoremstyle{remark}
\newtheorem{remark}[theorem]{Remark}

\usepackage{enumitem}
\setlist[enumerate,1]{label=(\arabic*)}
\setlist[enumerate,2]{label=(\roman*)}

\numberwithin{equation}{section}
\numberwithin{figure}{section}
\numberwithin{table}{section}

\usepackage{pgfplots}
\pgfplotsset{compat=1.17}

\usepackage{multirow}

\usepackage{here}

\usepackage{hyperref}

\usepackage[style=alphabetic,isbn=false,url=false,doi=false,giveninits=true,maxnames=9,sorting=nyt,backend=bibtex]{biblatex} %arXiv
\renewbibmacro{in:}{}
\bibliography{reference.bib}

\title{An inequality for the convolutions on unimodular locally compact groups and the optimal constant of Young's inequality}
\author{Takashi Satomi}
\date{\today} 

\begin{document}

\maketitle

\begin{abstract}

  Let $\mu$ be the Haar measure of a unimodular locally compact group $G$ and $m (G)$ as the infimum of the volumes of all open subgroups of $G$.
  The main result of this paper is that
  \begin{align*}
    \int_{G}^{} f \circ \left( \phi_1 * \phi_2 \right) \left( g \right) dg \leq \int_{\mathbb{R}}^{} f \circ \left( \phi_1^* * \phi_2^* \right) \left( x \right) dx
  \end{align*}
  holds for any measurable functions $\phi_1, \phi_2 \colon G \to \mathbb{R}_{\geq 0}$ with $\mu ( \mathrm{supp} \; \phi_1 ) + \mu ( \mathrm{supp} \; \phi_2 ) \leq m(G)$ and any convex function $f \colon \mathbb{R}_{\geq 0} \to \mathbb{R}$ with $f(0) = 0$.
  Here $\phi^*$ is the rearrangement of $\phi$.

  Let $Y_O(P,G)$ and $Y_R(P,G)$ denote the optimal constants of Young's and the reverse Young's inequality, respectively, under the assumption $\mu ( \mathrm{supp} \; \phi_1 ) + \mu ( \mathrm{supp} \; \phi_2 ) \leq m(G)$.
  Then we have $Y_O(P,G) \leq Y_O(P,\mathbb{R})$ and $Y_R(P,G) \geq Y_R(P,\mathbb{R})$ as a corollary.
  Thus, we obtain that $m (G) = \infty$ if and only if $H (p,G) \leq H (p, \mathbb{R})$ in the case of $p' := p/(p-1) \in 2 \mathbb{Z}$, where $H (p,G)$ is the optimal constant of the Hausdorff--Young inequality.

\end{abstract}



\section{Introduction}

The main result of this paper is a generalization of an inequality for the convolution and the rearrangement by Wang--Madiman \cite[Theorem 7.4]{MR3252379} to any unimodular locally compact group $G$ (Theorem \ref{thm:arrange-convex-inequality}).
As a corollary, the optimal constants of Young's inequality (Corollary \ref{cor:Young-Fournier} \ref{item:Young-Fournier-obverse}) and the Hausdorff--Young inequality (Corollary \ref{cor:Hausdorff-Young-mG-infinite}) are bounded from above, and the optimal constant of the reverse Young's inequality is bounded from below (Corollary \ref{cor:Young-Fournier} \ref{item:Young-Fournier-reverse}).
These inequalities are best possible for $G = \mathbb{R}$.

The rearrangement of a measurable function is defined as follows.

\begin{definition}
  \label{def:rearrange}

  The rearrangement $\phi^* \colon \mathbb{R} \to \mathbb{R}_{\geq 0}$ of a measurable function $\phi \colon G \to \mathbb{R}_{\geq 0}$ on a measure space $(G , \mu )$ is defined as
  \begin{align*}
    \phi^* (x) := \inf \{ t \geq 0 \mid \mu ( \phi^{-1} ( \mathbb{R}_{> t})) \leq 2 |x| \}.
  \end{align*}

\end{definition}

For $G = \mathbb{R}$, the function $\phi^*$ is called the symmetric decreasing rearrangement of $\phi$.
Definition \ref{def:rearrange} was essentially defined by O'Neil \cite[Definition 1.2]{MR146673} to generalize the Lorentz space \cite[Section 1]{MR28925} to any measurable space $G$.

Now, we write $\mathrm{supp} \; \phi := \{ g \in G \mid \phi (g) \neq 0 \}$ (we assume that the definition of measurable function includes the assumption that $\mathrm{supp} \; \phi$ is $\sigma$-finite).
We denote by $\mu$ the Haar measure of a unimodular locally compact group $G$ and by $m (G)$ the infimum of the volumes of all open subgroups of $G$.
The main theorem of this paper is as follows.

\begin{theorem}
  \label{thm:arrange-convex-inequality}

  Suppose measurable functions $\phi_1,\phi_2 \colon G \to \mathbb{R}_{\geq 0}$ on a unimodular locally compact group $G$ satisfy
  \begin{align}
    \mu ( \mathrm{supp} \; \phi_1 ) + \mu ( \mathrm{supp} \; \phi_2 ) & \leq m(G), \label{eq:supp-small}                                     \\
    \phi_1 * \phi_2 (g) , \phi_1^* * \phi_2^* (x)                     & < \infty \; \text{a.e.}, \label{eq:arrange-convex-inequality-finite}
  \end{align}
  and a convex function $f \colon \mathbb{R}_{\geq 0} \to \mathbb{R}$ satisfies $f(0) = 0$.
  If $\int_{G}^{} f \circ \left( \phi_1 * \phi_2 \right) \left( g \right) dg$ and $\int_{\mathbb{R}}^{} f \circ \left( \phi_1^* * \phi_2^* \right) \left( x \right) dx$ can be defined, then we have
  \begin{align}
    \int_{G}^{} f \circ \left( \phi_1 * \phi_2 \right) \left( g \right) dg \leq \int_{\mathbb{R}}^{} f \circ \left( \phi_1^* * \phi_2^* \right) \left( x \right) dx. \label{eq:arrange-convex-inequality-main}
  \end{align}

\end{theorem}

We note that $m (G) = \infty$ if and only if $G$ has no open compact subgroups (Remark \ref{rem:m-finite}).
In this case, Theorem \ref{thm:arrange-convex-inequality} holds even when $\mu ( \mathrm{supp} \; \phi_1) = \infty$ or $\mu ( \mathrm{supp} \; \phi_2) = \infty$.

Theorem \ref{thm:arrange-convex-inequality} was proved in the case of $G = \mathbb{R}$ or some cases of $f$ as Table \ref{tab:arrange-convex-inequality-compare}.
For the Brunn--Minkowski inequality and Kemperman's result, see the previous paper \cite[Corollary 1.3]{satomi2021inequality} of the author.
Now, we consider the case of $G = \mathbb{R}$.
For $f (y) := - y^p$ with $0 < p \leq 1$, Theorem \ref{thm:arrange-convex-inequality} was proved by Brascamp--Lieb \cite[Proposition 9]{MR412366} to improve the reverse Young's inequality (Fact \ref{fact:Beckner-Brascamp-Lieb} \ref{item:Beckner-Brascamp-Lieb-reverse}) and the Prékopa--Leindler inequality.
When $\phi_1$ and $\phi_2$ are integrable and $f$ is any convex function, Theorem \ref{thm:arrange-convex-inequality} was essentially proved by Wang--Madiman to bound Shannon's differential entropy of $\phi_1 * \phi_2$ from below \cite[Theorem 7.4]{MR3252379}.

\begin{table}
  \centering
  \renewcommand{\arraystretch}{1.4}
  \caption{Comparison of Theorem \ref{thm:arrange-convex-inequality} with some known results}
  \label{tab:arrange-convex-inequality-compare}
  \begin{tabular}{c||c|c}
    $f (y) \quad \backslash \quad G$ & $G = \mathbb{R}$                                                & $G$: unimodular locally compact group                        \\
    \hline \hline
    $\left\{
      \begin{aligned}
         & 0  &  & \text{if} \; y = 0 \\
         & -1 &  & \text{if} \; y > 0
      \end{aligned}
    \right.$                         & the Brunn--Minkowski inequality                                 & Kemperman \cite[Theorem 1.2]{MR202913}                       \\
    \hline
    $-y^p \; (0 < p \leq 1)$         & Brascamp--Lieb \cite[Proposition 9]{MR412366}                   & \multirow{3}{*}{Theorem \ref{thm:arrange-convex-inequality}} \\
    \cline{1-2}
    positive convex                  & Burchard \cite[Corollary 1]{MR2691260}                          &                                                              \\
    \cline{1-2}
    convex function                  & Wang--Madiman \cite[Theorem 7.4]{MR3252379}                     &                                                              \\
    \hline
    not convex                       & \multicolumn{2}{c}{not hold (Section \ref{sec:example-convex})}
  \end{tabular}
\end{table}

In the case of $G = \mathbb{R}^n$ for $n \geq 2$, the symmetric decreasing rearrangement $\phi^\star \colon \mathbb{R}^n \to \mathbb{R}_{\geq 0}$ defined in Wang--Madiman's paper differs from $\phi^*$ in Definition \ref{def:rearrange}.
Nevertheless, there is a one-to-one correspondence between $\phi^*$ and $\phi^\star$.
In this case, one can see that Wang--Madiman's result is sharper than Theorem \ref{thm:arrange-convex-inequality} by replacing $\phi_1$ and $\phi_2$ in Theorem \ref{thm:arrange-convex-inequality} with $\phi_1^\star$ and $\phi_2^\star$, respectively.

When $\phi_1$ and $\phi_2$ are characteristic functions on a unimodular locally compact group $G$, Theorem \ref{thm:arrange-convex-inequality} was previously obtained by the author \cite[Theorem 1.1]{satomi2021inequality}.
By using this result and the layer cake representation (Section \ref{sec:rearrange}), this result is generalized to any measurable functions $\phi_1$ and $\phi_2$ as in Theorem \ref{thm:arrange-convex-inequality}.

By Theorem \ref{thm:arrange-convex-inequality}, the integral for the convolution on $G$ is bounded by the integral on $\mathbb{R}$.
Thus, some inequalities on $G$ can be obtained from inequalities on $\mathbb{R}$.
For example, we bound the optimal constant of Young's inequality from above and the reverse Young's inequality from below (Corollary \ref{cor:Young-Fournier}).
Let
\begin{align*}
  B(p) := \frac{p^{1/2p}}{|p'|^{1/2p'}}, \quad
  \frac{1}{p} + \frac{1}{p'} = 1
\end{align*}
and $\| \cdot \|_p$ denotes the $L^p$-norm on $G$ for $p \in \mathbb{R}_{> 0}$.
For $p_1,p_2 > 0$ with
\begin{align}
  \frac{1}{p_1} + \frac{1}{p_2} > 1, \label{eq:p1p2-small}
\end{align}
we let the real number $q (p_1,p_2) > 0$ such that
\begin{align}
  \frac{1}{q (p_1,p_2)} = \frac{1}{p_1} + \frac{1}{p_2} - 1. \label{eq:Young-p-relate}
\end{align}
We note that $q (p_1,p_2) > 1$ holds for any $p_1,p_2 > 1$ with \eqref{eq:p1p2-small}.
We have \eqref{eq:p1p2-small} and $q (p_1,p_2) < 1$ for any $0 < p_1,p_2 < 1$.
Theorem \ref{thm:arrange-convex-inequality} shows the following corollary.

\begin{corollary}
  \label{cor:Young-Fournier}

  Suppose measurable functions $\phi_1 , \phi_2 \colon G \to \mathbb{R}_{\geq 0}$ on a unimodular locally compact group $G$ satisfy \eqref{eq:supp-small}.

  \begin{enumerate}
    \item \label{item:Young-Fournier-obverse}
          (a stronger version of Young's inequality)
          We have
          \begin{align*}
            \| \phi_1 * \phi_2 \|_{q (p_1,p_2)}
            \leq \frac{B(p_1) B(p_2)}{B \circ q (p_1,p_2)} \| \phi_1 \|_{p_1} \| \phi_2 \|_{p_2}
          \end{align*}
          for any $p_1,p_2 > 1$ with \eqref{eq:p1p2-small}.

    \item \label{item:Young-Fournier-reverse}
          (a stronger version of the reverse Young's inequality)
          We have
          \begin{align*}
            \| \phi_1 * \phi_2 \|_{q (p_1,p_2)}
            \geq \frac{B(p_1) B(p_2)}{B \circ q (p_1,p_2)} \| \phi_1 \|_{p_1} \| \phi_2 \|_{p_2}
          \end{align*}
          for any $0 < p_1,p_2 < 1$.

  \end{enumerate}

\end{corollary}

When $m (G) = \infty$, Corollary \ref{cor:Young-Fournier} can be regarded as an extension of Fournier's result \cite[Theorem 1]{MR461034} (Subsection \ref{subsec:Young-function-2}).
In this case, we have an upper bound of the optimal constant of the Hausdorff--Young inequality (Corollary \ref{cor:Hausdorff-Young-mG-infinite}) by using Corollary \ref{cor:Young-Fournier} \ref{item:Young-Fournier-obverse} and a result of Klein--Russo (Fact \ref{fact:Young-Hausdorff-Young-relate}).
It was proved that Corollary \ref{cor:Young-Fournier} is best possible for $G = \mathbb{R}$ (Fact \ref{fact:Beckner-Brascamp-Lieb}).
Although Corollary \ref{cor:Young-Fournier} was already proved in particular cases of $G$, Corollary \ref{cor:Young-Fournier} includes new cases such as compact groups or semisimple Lie groups as far as the author knows (Section \ref{sec:Young}).

Here is the organization of this paper.
In Section \ref{sec:Young}, we will compare Corollary \ref{cor:Young-Fournier} with some known results and show Corollary \ref{cor:Young-Fournier} by using Theorem \ref{thm:arrange-convex-inequality}.
In Section \ref{sec:Hausdorff-Young}, we will bound the optimal constant of the Hausdorff--Young inequality from above by using Corollary \ref{cor:Young-Fournier} \ref{item:Young-Fournier-obverse}.
In Section \ref{sec:example-convex}, we will see that the function $f$ in Theorem \ref{thm:arrange-convex-inequality} must be a convex function for $G = \mathbb{R}$.
In Section \ref{sec:rearrange}, we will summarize some properties of the rearrangement to show Theorem \ref{thm:arrange-convex-inequality}.
Sections \ref{sec:proof-ft}-\ref{sec:rearrange-proof-general} will be devoted to the proof of Theorem \ref{thm:arrange-convex-inequality}.
In Section \ref{sec:proof-ft}, we will show Theorem \ref{thm:arrange-convex-inequality} for certain convex functions $f = f_t$.
In Section \ref{sec:integrable-proof}, we will show Theorem \ref{thm:arrange-convex-inequality} in the case where $\phi_1 * \phi_2$ is integrable.
In Section \ref{sec:rearrange-proof-general}, we will complete the proof of Theorem \ref{thm:arrange-convex-inequality}.

\section{The optimal constant of (the reverse) Young's inequality}
\label{sec:Young}

In this section, we compare Corollary \ref{cor:Young-Fournier} with some known results, and show Corollary \ref{cor:Young-Fournier} by using Theorem \ref{thm:arrange-convex-inequality}.
Here is a summary of this section.
We define the optimal constants of Young's inequality (Fact \ref{fact:Young}) and the reverse Young's inequality (Theorem \ref{thm:reverse-Young}) as $\tilde{Y}_O (P,G)$ and $\tilde{Y}_R (P,G)$, respectively, for $P = (p_1 , p_2) \in \mathbb{R}_{> 0}^2$.
Similarly, we define these optimal constants as $Y_O (P,G)$ and $Y_R (P,G)$ under the assumption \eqref{eq:supp-small}.
We will define $\tilde{Y}_O (P,G)$, $\tilde{Y}_R (P,G)$, $Y_O (P,G)$ and $Y_R (P,G)$ precisely in Definition \ref{def:Young-optimal} later.
Similarly, we denote by $H(p,G)$ the optimal constant of the Hausdorff--Young inequality for $1 < p \leq 2$ (Definition \ref{def:Russo}).
Some known results about the optimal constants are summarized in Table \ref{tab:optimal-constant}.

\begin{table}[H]
  \centering
  \caption{Some known results about the optimal constants}
  \label{tab:optimal-constant}
  \renewcommand{\arraystretch}{1.4}
  \begin{tabular}{c|c|c}
    $m (G) < \infty$                                                        & $G = \mathbb{R}^n$                                                                       & $m (G) = \infty$                                                     \\
    \hline \hline
    $Y_O (P,G) \leq C (P)$                                                  &                                                                                          &                                                                      \\
    Corollary \ref{cor:Young-Fournier} \ref{item:Young-Fournier-obverse}    & $Y_O (P,G) = \tilde{Y}_O (P,G) = C(P)^n$                                                 & $Y_O (P,G) = \tilde{Y}_O (P,G) \leq C (P)$                           \\
    \cline{1-1}
    $\tilde{Y}_O (P,G) = 1$                                                 & Beckner (Fact \ref{fact:Beckner-Brascamp-Lieb} \ref{item:Beckner-Brascamp-Lieb-obverse}) & Corollary \ref{cor:Young-Fournier} \ref{item:Young-Fournier-obverse} \\
    Fournier \cite[Theorem 3]{MR461034}                                     &                                                                                          &                                                                      \\
    \hline
    $Y_R (P,G) \geq C (P)$                                                  &                                                                                          &                                                                      \\
    Corollary \ref{cor:Young-Fournier} \ref{item:Young-Fournier-reverse}    & $Y_R (P,G) = \tilde{Y}_R (P,G) = C(P)^n$                                                 & $Y_R (P,G) = \tilde{Y}_R (P,G) \geq C (P)$                           \\
    \cline{1-1}
    $\tilde{Y}_R (P,G) = 1$                                                 & Brascamp--Lieb \cite{MR412366} ($n = 1$),                                                & Corollary \ref{cor:Young-Fournier} \ref{item:Young-Fournier-reverse} \\
    Theorem \ref{thm:reverse-Young}, Remark \ref{rem:Y-tilde-equality-hold} & Barthe (Fact \ref{fact:Beckner-Brascamp-Lieb} \ref{item:Beckner-Brascamp-Lieb-reverse})  &                                                                      \\
    \hline
                                                                            &                                                                                          & $H (p,G) \leq B (p)$ if $p' \in 2 \mathbb{Z}$                        \\
    $H (p,G) = 1$                                                           & $H (p,G) = B (p)^n$                                                                      & Russo \cite[Theorem 3 (b)]{MR515223},                                \\
    Russo \cite[Theorem 1]{MR435731}                                        & Babenko \cite{MR0138939} ($p' \in 2 \mathbb{Z}$),                                        & Corollary \ref{cor:Hausdorff-Young-mG-infinite}                      \\
    \cline{3-3}
                                                                            & Beckner \cite[Theorem 1]{MR385456}                                                       & $H (p,G) < 1$                                                        \\
                                                                            &                                                                                          & Fournier \cite[Theorem 2]{MR461034}
  \end{tabular}
\end{table}

Corollary \ref{cor:Young-Fournier} implies
\begin{align*}
  Y_O (P,G)
  \leq C (P), \quad
  Y_R (P,G)
  \geq C (P), \quad
  C (P)
  := \frac{B(p_1) B(p_2)}{B \circ q (p_1,p_2)}.
\end{align*}

When $m (G) = \infty$, we have $Y_O (P,G) = \tilde{Y}_O (P,G)$ and $Y_R (P,G) = \tilde{Y}_R (P,G)$ by definition.
For $G = \mathbb{R}^n$, the constants $Y_O (P,G)$ and $Y_R (P,G)$ were given explicitly (Fact \ref{fact:Beckner-Brascamp-Lieb}).
Corollary \ref{cor:Young-Fournier} can be proved in some cases of $G$ by using a result of Cowling--Martini--M\"{u}ller--Parcet (Fact \ref{fact:Cowling-Martini-Muller-Parcet}).
On the other hand, when $G$ is a semisimple Lie group, Corollary \ref{cor:Young-Fournier} has not been known as far as the author knows.

When $m (G) < \infty$, we have $\tilde{Y}_O (P,G) = 1$ and $\tilde{Y}_R (P,G) = 1$ (Remark \ref{rem:Y-tilde-equality-hold}).
So, we consider $Y_O (P,G)$ and $Y_R (P,G)$ instead of $\tilde{Y}_O (P,G)$ and $\tilde{Y}_R (P,G)$.
When $G$ is a Lie group and both of $\mathrm{supp} \; \phi_1$ and $\mathrm{supp} \; \phi_2$ are sufficiently small (even smaller than ones in (1.1)), the optimal constant is determined only by $\dim G$ (Fact \ref{fact:Bennett-Bez-Buschenhenke-Cowling-Flock}).
This fact gives a stronger upper bound than Corollary \ref{cor:Young-Fournier}.
On the other hand, there are more functions which can be applied to Corollary \ref{cor:Young-Fournier}.

In Subsection \ref{subsec:Young-optimal-constant}, we will define the optimal constants $\tilde{Y}_O (P,G)$, $\tilde{Y}_R (P,G)$, $Y_O (P,G)$, and $Y_R (P,G)$.
In Subsection \ref{subsec:Young-function-2}, we will compare Corollary \ref{cor:Young-Fournier} with some known results for $N = 2$.
In Subsection \ref{subsec:Young-function-more-2}, we will see that similar results are valid for any $N$.
In Subsection \ref{subsec:Young-Fourier-transform is equal to proof}, we will show Corollary \ref{cor:Young-Fournier} by using Theorem \ref{thm:arrange-convex-inequality}.

\subsection{The definition of the optimal constants}
\label{subsec:Young-optimal-constant}

In this subsection, we define the optimal constant of (the reverse) Young's inequality.

\begin{definition}
  \label{def:Young-optimal}

  Let $N \in \mathbb{Z}_{\geq 1}$.
  For $P := (p_1, p_2, \cdots , p_N ) \in \mathbb{R}_{> 0}^N$ with
  \begin{align*}
    \sum_{k = 1}^N \frac{1}{p_k} > N - 1,
  \end{align*}
  we let the real number $q (P) > 0$ such that
  \begin{align*}
    \frac{1}{q ( P )} = 1 - N + \sum_{k = 1}^N \frac{1}{p_k}.
  \end{align*}

  \begin{enumerate}
    \item
          When $p_1, p_2, \cdots , p_N , q (P) \neq 1$, we define
          \begin{align*}
            C (P) := \frac{1}{B \circ q (P)} \prod_{k = 1}^N B (p_k).
          \end{align*}

    \item \label{item:Young-optimal-tilde}
          For a unimodular locally compact group $G$, we define
          \begin{align*}
            \tilde{\mathcal{B}} := \{ ( \phi_1 , \phi_2 , \cdots , \phi_N ) \mid \phi_1, \phi_2, \cdots , \phi_N \colon G \to \mathbb{R}_{\geq 0} , \; \| \phi_1 \|_{p_1} = \| \phi_2 \|_{p_2} = \cdots = \| \phi_N \|_{p_N} = 1 \}
          \end{align*}
          and
          \begin{align*}
            \tilde{Y}_O (P, G) & := \sup \{ \| \phi_1 * \phi_2 * \cdots * \phi_N \|_{q (P)} \mid ( \phi_1 , \phi_2 , \cdots , \phi_N ) \in \tilde{\mathcal{B}} \}, \\
            \tilde{Y}_R (P, G) & := \inf \{ \| \phi_1 * \phi_2 * \cdots * \phi_N \|_{q (P)} \mid ( \phi_1 , \phi_2 , \cdots , \phi_N ) \in \tilde{\mathcal{B}} \}.
          \end{align*}

    \item
          Let $G$ and $\tilde{\mathcal{B}}$ be as in \ref{item:Young-optimal-tilde}.
          For $N = 2$, we define $\mathcal{B}$ as the set of all elements of $\tilde{\mathcal{B}}$ with \eqref{eq:supp-small} and
          \begin{align*}
            Y_O (P, G) & := \sup \{ \| \phi_1 * \phi_2 \|_{q (P)} \mid ( \phi_1 , \phi_2 ) \in \mathcal{B} \}, \\
            Y_R (P, G) & := \inf \{ \| \phi_1 * \phi_2 \|_{q (P)} \mid ( \phi_1 , \phi_2 ) \in \mathcal{B} \}.
          \end{align*}

  \end{enumerate}

\end{definition}

We note that $C(P) < 1 < q(P)$ holds for any $p_1 , p_2, \cdots , p_N > 1$ and $C(P) > 1 > q (P)$ holds for any $0 < p_1 , p_2, \cdots , p_N < 1$.
When $m (G) = \infty$, the assumption \eqref{eq:supp-small} always holds.
Thus, for any unimodular locally compact group $G$, we have
\begin{align}
  Y_O (P , G) \leq \tilde{Y}_O (P , G), \quad
  Y_R (P , G) \geq \tilde{Y}_R (P , G) \label{eq:Y-tilde-Y-inequality}
\end{align}
and these equalities hold if $m (G) = \infty$.
Actually, these equalities hold if and only if $m (G) = \infty$ (Remark \ref{rem:Y-tilde-equality-hold}).

\begin{remark}[\cite{satomi2021inequality}, Remark 2.4 (3)]
  \label{rem:m-finite}

  The following three conditions \ref{item:m-finite-finite}, \ref{item:m-finite-open-compact}, and \ref{item:m-finite-G0-compact} are equivalent for any locally compact group $G$.

  \begin{enumerate}
    \item \label{item:m-finite-finite}
          One has $m (G) < \infty$.

    \item \label{item:m-finite-open-compact}
          The locally compact group $G$ has an open compact subgroup.

    \item \label{item:m-finite-G0-compact}
          The identity component $G_0$ of $G$ is compact.

  \end{enumerate}

  If the equivalent conditions \ref{item:m-finite-finite}, \ref{item:m-finite-open-compact}, and \ref{item:m-finite-G0-compact} are satisfied, then we have $m (G) = \mu (G_0)$.
  Thus, \eqref{eq:supp-small} holds if and only if at least one of the following two conditions \ref{item:m-finite-vol-G0} and \ref{item:m-finite-G0-non-compact} are satisfied.

  \begin{enumerate}[label=(\alph*)]
    \item \label{item:m-finite-vol-G0}
          One has $\mu ( \mathrm{supp} \; \phi_1 ) + \mu ( \mathrm{supp} \; \phi_2 ) \leq \mu (G_0)$.

    \item \label{item:m-finite-G0-non-compact}
          The identity component $G_0$ is not compact.

  \end{enumerate}

  In particular, if $G_0$ is open (e.g., $G$ is a Lie group), then \eqref{eq:supp-small} and \ref{item:m-finite-vol-G0} are equivalent.

\end{remark}

\subsection{The case of two functions}
\label{subsec:Young-function-2}

In the case of $N = 1$, we have $q (p_1) = p_1$ and hence
\begin{align*}
  Y_O (p_1 , G)
  = \tilde{Y}_O (p_1 , G)
  = Y_R (p_1 , G)
  = \tilde{Y}_R (p_1 , G)
  = C (p_1)
  = 1.
\end{align*}
Thus, the case of $N = 2$ is the simplest example such that these constants are not trivial.
In this subsection, we consider the case of $N = 2$ in Definition \ref{def:Young-optimal}.
We have the following statements.

\begin{fact}[Young's inequality \cite{MR0005741}]
  \label{fact:Young}

  Let $G$, $P = (p_1 , p_2) \in \mathbb{R}_{\geq 1}^2$, and $\tilde{Y}_O (P , G)$ be as in Definition \ref{def:Young-optimal} for $N = 2$.
  Then we have $\tilde{Y}_O (P , G) \leq 1$.

\end{fact}

\begin{theorem}[The reverse Young's inequality]
  \label{thm:reverse-Young}

  Let $G$, $P = (p_1 , p_2) \in (0,1]^2$, and $\tilde{Y}_R (P , G)$ be as in Definition \ref{def:Young-optimal} for $N = 2$.
  Then we have $\tilde{Y}_R (P , G) \geq 1$.

\end{theorem}

In the case of $G = \mathbb{R}$ or $G= S^1$, Theorem \ref{thm:reverse-Young} was proved by Leindler \cite[Theorems 1 and 2]{MR0430188}.

\begin{proof}[Proof of Theorem \ref{thm:reverse-Young}]

  It suffices to show $\| \phi_1 * \phi_2 \|_p^p \geq 1$ for any measurable functions $\phi_1 , \phi_2 \colon G \to \mathbb{R}_{\geq 0}$ with $\| \phi_1 \|_{p_1} = \| \phi_2 \|_{p_2} = 1$.
  Since
  \begin{gather*}
    \frac{1}{p_1'} + \frac{1}{p_2'} + \frac{1}{p} = 1 , \quad
    \phi_1 (g) \phi_2 (g^{-1} g')
    = \phi_1 (g)^{p_1/p_2'} \phi_2 (g^{-1} g')^{p_2/p_1'} ( \phi_1 (g)^{p_1}\phi_2 (g^{-1} g')^{p_2} )^{1/p}
  \end{gather*}
  for any $g , g' \in G$ by \eqref{eq:Young-p-relate}, we have
  \begin{align*}
    \phi_1 * \phi_2 (g')
    \geq \| \phi_1 \|_{p_1}^{p_1 / p_2'} \| \phi_2 \|_{p_2}^{p_2 / p_1'} \left( \int_{G}^{} \phi_1 (g)^{p_1} \phi_2 (g^{-1} g')^{p_2} dg  \right)^{1/p}
    = \left( \int_{G}^{} \phi_1 (g)^{p_1} \phi_2 (g^{-1} g')^{p_2} dg  \right)^{1/p}
  \end{align*}
  by H\"{o}lder's inequality.
  Thus,
  \begin{align*}
    \| \phi_1 * \phi_2 \|_p^p
    \geq \int_{G}^{} \int_{G}^{} \phi_1 (g)^{p_1} \phi_2 (g^{-1} g')^{p_2} dg dg'
    = \| \phi_1 \|_{p_1}^{p_1} \| \phi_2 \|_{p_2}^{p_2}
    = 1
  \end{align*}
  is obtained.
\end{proof}

\begin{remark}
  \label{rem:Y-tilde-equality-hold}

  When $m (G) < \infty$, the equalities in Fact \ref{fact:Young} and Theorem \ref{thm:reverse-Young} hold.
  In fact, there is an open compact subgroup $G' \subset G$ by Remark \ref{rem:m-finite} in this case.
  Thus, the equalities in Fact \ref{fact:Young} and Theorem \ref{thm:reverse-Young} hold for $\phi_1 = \phi_2 = 1_{G'}$, where $1_{G'}$ is a characteristic function of $G'$.
  On the other hand, we have
  \begin{align*}
    Y_O (P , G) \leq C(P) < 1 = \tilde{Y}_O (P,G)
  \end{align*}
  for $p_1 , p_2 > 1$ by Corollary \ref{cor:Young-Fournier} \ref{item:Young-Fournier-obverse} and similarly
  \begin{align*}
    Y_R (P , G) \geq C(P) > 1 = \tilde{Y}_R (P,G)
  \end{align*}
  for $p_1 , p_2 < 1$ by Corollary \ref{cor:Young-Fournier} \ref{item:Young-Fournier-reverse}.
  That is, the equality of \eqref{eq:Y-tilde-Y-inequality} does not hold when $m (G) < \infty$.
  Thus, the equality of \eqref{eq:Y-tilde-Y-inequality} holds if and only if $m (G) = \infty$.

\end{remark}

We fix $P = (p_1,p_2)$ with $p_1,p_2 >1$.
Fournier proved that there is a real number $c (P) < 1$, which is independent in $G$, such that $Y_O (P , G) \leq c (P)$ for any $G$ with $m(G) = \infty$ \cite[Theorem 1]{MR461034}.
Corollary \ref{cor:Young-Fournier} \ref{item:Young-Fournier-obverse} implies that The minimum of $c (P)$ satisfying this claim is $C (P)$.

For $G = \mathbb{R}^n$, Beckner and Brascamp--Lieb explicitly determined the values of $Y_O (P , G)$ and $Y_R (P , G)$, respectively, as follows.

\begin{fact}
  \label{fact:Beckner-Brascamp-Lieb}

  Let $G$, $P = (p_1, p_2)$, $Y_O (P, G)$, $Y_R (P , G)$, and $C(P)$ be as in Definition \ref{def:Young-optimal} for $N = 2$, and $n \in \mathbb{Z}_{\geq 1}$.

  \begin{enumerate}
    \item \label{item:Beckner-Brascamp-Lieb-obverse}
          (Beckner, \cite[Theorem 3]{MR385456})
          If $p_1,p_2 > 1$, then we have $Y_O ( P , \mathbb{R}^n ) = C(P)^n$.

    \item \label{item:Beckner-Brascamp-Lieb-reverse}
          (Barthe, \cite[Theorem 1]{MR1616143})
          If $p_1,p_2 < 1$, then we have $Y_R (P , \mathbb{R}^n) = C(P)^n$.

  \end{enumerate}

\end{fact}

Fact \ref{fact:Beckner-Brascamp-Lieb} \ref{item:Beckner-Brascamp-Lieb-reverse} was proved by Brascamp--Lieb when $n = 1$ \cite[Theorem 8]{MR412366}.

\begin{remark}

  There are some proofs of Fact \ref{fact:Beckner-Brascamp-Lieb} \ref{item:Beckner-Brascamp-Lieb-obverse}.

  \begin{enumerate}
    \item
          Beckner proved it by using the fact that pairs of Gaussian functions are solutions of functional equations which are necessary for $\| \phi_1^* * \phi_2^* \|_{q (P)}$ to be the maximum \cite[Theorem 3]{MR385456}.

    \item
          Brascamp--Lieb proved it by showing $Y_O ( P, \mathbb{R}^n ) = Y_O ( P ,\mathbb{R} )^n$ and by observing the behavior of the limit of an upper bound for $Y_O ( P, \mathbb{R}^n )$ as $n \to \infty$ \cite[Section 2.5]{MR412366}.

    \item
          Barthe gave a direct proof, which utilizes the change of variable by Henstock--Macbeath \cite[Section 5]{MR56669} and the weighted AM-GM inequality \cite[Theorem 1]{MR1616143} (see also \cite{MR1650312}).

    \item
          Carlen--Lieb--Loss implicitly proved it by considering heat equations which have initial values $\phi_1^{p_1}$ and $\phi_2^{p_2}$, and by showing that the $L^{q (P)}$-norm of the convolution of these functions is increasing with respect to time \cite[Theorem 3.1]{MR2077162}.
          A similar argument appears in the papers of Bennett--Carbery--Christ--Tao \cite[Example 1.5]{MR2377493} and Bennett--Bez \cite[Section 1.1]{MR2496567}.

    \item
          Cordero-Erausquin--Ledoux proved it by estimating $\exp \circ Y_O ( P , \mathbb{R} )$ from above by a linear combination of Shannon's differential entropy of functions, and by calculating the Fisher information of these functions \cite[Theorem 6]{MR2644890}.

  \end{enumerate}

  Brascamp--Lieb \cite[Theorem 8]{MR412366}, Barthe \cite[Theorem 1]{MR1616143}, and Bennett--Bez \cite[Section 1.1]{MR2496567} also proved Fact \ref{fact:Beckner-Brascamp-Lieb} \ref{item:Beckner-Brascamp-Lieb-reverse} by similar arguments.

\end{remark}

By Fact \ref{fact:Beckner-Brascamp-Lieb}, Corollary \ref{cor:Young-Fournier} is best possible for $G = \mathbb{R}$.
Corollary \ref{cor:Young-Fournier} \ref{item:Young-Fournier-obverse} can be shown by using Fact \ref{fact:Beckner-Brascamp-Lieb} and the following fact in some cases of $G$.

\begin{fact}[Cowling--Martini--M\"{u}ller--Parcet, {\cite[Proposition 2.2 (iii)]{MR4000236}}]
  \label{fact:Cowling-Martini-Muller-Parcet}

  Let $G$, $P = (p_1 , p_2) \in \mathbb{R}_{\geq 1}^2$, and $\tilde{Y}_O (P , G)$ be as in Definition \ref{def:Young-optimal} for $N = 2$.
  Then we have $\tilde{Y}_O (P , G) \leq \tilde{Y}_O (P , G') \tilde{Y}_O (P , G / G')$ for any closed normal unimodular subgroup $G' \subset G$.

\end{fact}

When $G$ is expressed as $G = G' \rtimes G / G'$, Fact \ref{fact:Cowling-Martini-Muller-Parcet} was proved by Klein--Russo \cite[Lemma 2.4]{MR499945}.
We can show Corollary \ref{cor:Young-Fournier} \ref{item:Young-Fournier-obverse} for some cases of $G$ such as non-compact nilpotent Lie groups by Fact \ref{fact:Beckner-Brascamp-Lieb} \ref{item:Beckner-Brascamp-Lieb-obverse} and Fact \ref{fact:Cowling-Martini-Muller-Parcet} (see also \cite{MR1304346}).
On the other hand, when $G$ is a semisimple Lie group, Corollary \ref{cor:Young-Fournier} cannot be proved only by this argument.
As far as the author knows, Corollary \ref{cor:Young-Fournier} has not been known in some cases of $G$ such as semisimple Lie groups or compact groups.

When $G$ is a Lie group, the optimal constant of Young's inequality depends only on the dimension $\dim G$ under a stronger assumption than \eqref{eq:supp-small}.

\begin{fact}[Bennett--Bez--Buschenhenke--Cowling--Flock, {\cite[Corollary 2.4]{MR4173156}}]
  \label{fact:Bennett-Bez-Buschenhenke-Cowling-Flock}

  Let $P = (p_1,p_2) \in \mathbb{R}_{> 1}^2$ and $C(P)$ be as in Definition \ref{def:Young-optimal} for $N = 2$.
  Then for any unimodular Lie group $G$ and $Y > C (P)^{\dim G}$, there exists a non-empty open subset $U \subset G$ such that
  \begin{align}
    \| \phi_1 * \phi_2 \|_p \leq Y \| \phi_1 \|_{p_1} \| \phi_2 \|_{p_2} \label{eq:Bennett-Bez-Buschenhenke-Cowling-Flock-inequality}
  \end{align}
  holds for any measurable functions $\phi_1 , \phi_2 \colon G \to \mathbb{R}_{\geq 0}$ with $\mathrm{supp} \; \phi_1 , \mathrm{supp} \; \phi_2 \subset U$.

\end{fact}

In the original paper \cite[Corollary 2.4]{MR4173156}, Fact \ref{fact:Bennett-Bez-Buschenhenke-Cowling-Flock} was proved only when $G$ is connected.
However, this connectedness is not necessary because we can take $U$ such as $U \subset G_0$.
By the result of Cowling--Martini--M\"{u}ller--Parcet \cite[Proposition 2.4 (i)]{MR4000236}, the constant $C (P)^{\dim G}$ in Fact \ref{fact:Bennett-Bez-Buschenhenke-Cowling-Flock} is the best possible for any $G$.

Since the optimal constant of Fact \ref{fact:Bennett-Bez-Buschenhenke-Cowling-Flock} is smaller than that of Corollary \ref{cor:Young-Fournier} (when $\dim G \geq 2$), Fact \ref{fact:Bennett-Bez-Buschenhenke-Cowling-Flock} gives a stronger bound than Corollary \ref{cor:Young-Fournier}.
On the other hand, the assumption \eqref{eq:supp-small} in Corollary \ref{cor:Young-Fournier} is weaker than the assumption in Fact \ref{fact:Bennett-Bez-Buschenhenke-Cowling-Flock} and hence there are more functions which can be applied to Corollary \ref{cor:Young-Fournier}.
In this sense, Corollary \ref{cor:Young-Fournier} is a new result for some locally compact groups such as $G = S^1$ as far as the author knows.

To point out the difference between Corollary \ref{cor:Young-Fournier} and Fact \ref{fact:Bennett-Bez-Buschenhenke-Cowling-Flock} more precisely, we consider the case of $G = \mathbb{R} \times S^1$.
Fact \ref{fact:Bennett-Bez-Buschenhenke-Cowling-Flock} implies that \eqref{eq:Bennett-Bez-Buschenhenke-Cowling-Flock-inequality} holds for any $Y > C (P)^2$, $\phi_1$, and $\phi_2$ when $\mathrm{supp} \; \phi_1$ and $\mathrm{supp} \; \phi_2$ are sufficiently small.
On the other hand, Corollary \ref{cor:Young-Fournier} implies that \eqref{eq:Bennett-Bez-Buschenhenke-Cowling-Flock-inequality} holds for $Y = C (P)$ and any measurable functions $\phi_1$ and $\phi_2$.
In fact, we can find examples such that the equality of \eqref{eq:Bennett-Bez-Buschenhenke-Cowling-Flock-inequality} holds by considering $S^1$-invariant functions \cite[Section IV.5]{MR385456}.

\subsection{The case of more than two functions}
\label{subsec:Young-function-more-2}

For $N \geq 3$, similar statements hold by repeating the arguments in Subsection \ref{subsec:Young-function-2}.
For example, we consider the case of $N = 3$.
Since
\begin{align*}
  \frac{1}{q ( q (p_1 , p_2) , p_3)}
  = \frac{1}{q (p_1 , p_2)} + \frac{1}{p_3} - 1
  = \frac{1}{p_1} + \frac{1}{p_2} + \frac{1}{p_3} - 2
  = \frac{1}{q (p_1,p_2,p_3)},
\end{align*}
we have $q ( q (p_1 , p_2) , p_3) = q (p_1,p_2,p_3)$.
Thus,
\begin{align*}
  C (p_1,p_2,p_3)
  = \frac{B(p_1) B(p_2) B(p_3)}{B \circ q (p_1,p_2,p_3)}
  = \frac{B(p_1) B(p_2)}{B \circ q (p_1,p_2)} \cdot \frac{B \circ q (p_1 , p_2) B(p_3)}{B \circ q ( q (p_1,p_2) , p_3 )}
  = C (p_1,p_2) C (q (p_1 , p_2) , p_3)
\end{align*}
holds and hence we have similar statements by repeating the arguments in Subsection \ref{subsec:Young-function-2}.
Similarly, we can generalize these statements for any $N \in \mathbb{Z}_{\geq 1}$.

Corollary \ref{cor:Young-Fournier} can be generalized for any $N$ by weakening the assumption \eqref{eq:supp-small} to $m (G) = \infty$.
That is, we have the following statement by Corollary \ref{cor:Young-Fournier}.

\begin{corollary}[Corollary \ref{cor:Young-Fournier} for any $N$]
  \label{cor:Young-Fournier-N-general}

  Let $N$, $G$, $P$, $\tilde{Y}_O (P, G)$, $\tilde{Y}_R (P , G)$, and $C(P)$ be as in Definition \ref{def:Young-optimal} with $m (G) = \infty$.

  \begin{enumerate}
    \item \label{item:Young-Fournier-N-general-obverse}
          If $P \in \mathbb{R}_{> 1}^N$, then we have $\tilde{Y}_O (P, G) \leq C(P)$.

    \item
          If $P \in (0,1)^N$, then we have $\tilde{Y}_R (P, G) \geq C(P)$.

  \end{enumerate}

\end{corollary}

Similarly, the other statements in Subsection \ref{subsec:Young-function-2} can be generalized to $N \geq 3$.
That is, the following statements hold.

\begin{fact}[Fact \ref{fact:Young} for any $N$, Klein--Russo, {\cite[Corollary 2.3]{MR499945}}]

  Let $N$, $G$, $P \in \mathbb{R}_{\geq 1}^N$, and $\tilde{Y}_O (P , G)$ be as in Definition \ref{def:Young-optimal}.
  Then we have $\tilde{Y}_O (P , G) \leq 1$.

\end{fact}

\begin{corollary}[Theorem \ref{thm:reverse-Young} for any $N$]

  Let $N$, $G$, $P \in (0,1]^N$, and $\tilde{Y}_R (P , G)$ be as in Definition \ref{def:Young-optimal}.
  Then we have $\tilde{Y}_R (P , G) \geq 1$.

\end{corollary}

\begin{fact}[Fact \ref{fact:Beckner-Brascamp-Lieb} for any $N$]

  We let $N$, $G$, $P$, $\tilde{Y}_O (P, G)$, $\tilde{Y}_R (P , G)$, and $C(P)$ be as in Definition \ref{def:Young-optimal}, and $n \in \mathbb{Z}_{\geq 1}$.

  \begin{enumerate}
    \item
          (Beckner, \cite[Theorem 4]{MR385456})
          If $P \in \mathbb{R}_{> 1}^N$, then we have $\tilde{Y}_O ( P , \mathbb{R}^n ) = C(P)^n$.

    \item
          If $P \in (0,1)^N$, then we have $\tilde{Y}_R ( P , \mathbb{R}^n ) = C(P)^n$.

  \end{enumerate}

\end{fact}

\begin{corollary}[Fact \ref{fact:Bennett-Bez-Buschenhenke-Cowling-Flock} for any $N$]

  Let $N$, $P \in \mathbb{R}_{> 1}^N$, $q (P)$, and $C (P)$ be as in Definition \ref{def:Young-optimal}.
  Then for any unimodular Lie group $G$ and $Y > C (P)^{\dim G}$, there exists a non-empty open subset $U \subset G$ such that
  \begin{align*}
    \| \phi_1 * \phi_2 * \cdots * \phi_N \|_{q (P)} \leq Y \prod_{k = 1}^N \| \phi_k \|_{p_k}
  \end{align*}
  holds for any measurable functions $\phi_1 , \phi_2 , \cdots , \phi_N \colon G \to \mathbb{R}_{\geq 0}$ with $\mathrm{supp} \; \phi_1 , \mathrm{supp} \; \phi_2 , \cdots , \mathrm{supp} \; \phi_N \subset U$.

\end{corollary}

Fact \ref{fact:Cowling-Martini-Muller-Parcet} cannot be generalized for $N \geq 3$ only by repeating Fact \ref{fact:Cowling-Martini-Muller-Parcet}.
Nevertheless, Cowling--Martini--M\"{u}ller--Parcet proved Fact \ref{fact:Cowling-Martini-Muller-Parcet} for any $N$ as follows.

\begin{fact}[Fact \ref{fact:Cowling-Martini-Muller-Parcet} for any $N$, Cowling--Martini--M\"{u}ller--Parcet, {\cite[Proposition 2.2 (iii)]{MR4000236}}]

  Let $N$, $G$, $P \in \mathbb{R}_{> 1}^N$, and $\tilde{Y}_O (P , G)$ be as in Definition \ref{def:Young-optimal}.
  Then we have $\tilde{Y}_O (P , G) \leq \tilde{Y}_O (P , G') \tilde{Y}_O (P , G / G')$ for any closed normal unimodular subgroup $G' \subset G$.

\end{fact}

\subsection{A proof of Corollary \ref{cor:Young-Fournier}}
\label{subsec:Young-Fourier-transform is equal to proof}

In this subsection, we show Corollary \ref{cor:Young-Fournier} by Theorem \ref{thm:arrange-convex-inequality}.

\begin{proof}[Proof of Corollary \ref{cor:Young-Fournier}]

  \begin{enumerate}
    \item
          We have $\| \phi_1 * \phi_2 \|_{q (p_1,p_2)} \leq \| \phi_1^* * \phi_2^* \|_{q (p_1,p_2)}$ by applying $f(y) := y^{q (p_1,p_2)}$ to Theorem \ref{thm:arrange-convex-inequality}.
          Since $\| \phi_1^* \|_{p_1} = \| \phi_1 \|_{p_1}$ and $\| \phi_2^* \|_{p_2} = \| \phi_2 \|_{p_2}$, we obtain
          \begin{align*}
            \| \phi_1 * \phi_2 \|_{q (p_1,p_2)}
            \leq \| \phi_1^* * \phi_2^* \|_{q (p_1,p_2)}
            \leq C (p_1,p_2) \| \phi_1^* \|_{p_1} \| \phi_2^* \|_{p_2}
            = C (p_1,p_2) \| \phi_1 \|_{p_1} \| \phi_2 \|_{p_2}
          \end{align*}
          by Fact \ref{fact:Beckner-Brascamp-Lieb} \ref{item:Beckner-Brascamp-Lieb-obverse}.

    \item
          We have $\| \phi_1 * \phi_2 \|_{q (p_1,p_2)} \geq \| \phi_1^* * \phi_2^* \|_{q (p_1,p_2)}$ by applying $f(y) := - y^{q (p_1,p_2)}$ to Theorem \ref{thm:arrange-convex-inequality}.
          Since $\| \phi_1^* \|_{p_1} = \| \phi_1 \|_{p_1}$ and $\| \phi_2^* \|_{p_2} = \| \phi_2 \|_{p_2}$, we obtain
          \begin{align*}
            \| \phi_1 * \phi_2 \|_{q (p_1,p_2)}
            \geq \| \phi_1^* * \phi_2^* \|_{q (p_1,p_2)}
            \geq C (p_1,p_2) \| \phi_1^* \|_{p_1} \| \phi_2^* \|_{p_2}
            = C (p_1,p_2) \| \phi_1 \|_{p_1} \| \phi_2 \|_{p_2}
          \end{align*}
          by Fact \ref{fact:Beckner-Brascamp-Lieb} \ref{item:Beckner-Brascamp-Lieb-reverse}.
          \qedhere

  \end{enumerate}

\end{proof}

\section{The Hausdorff--Young inequality}
\label{sec:Hausdorff-Young}

In this section, we give an upper bound of the optimal constant $H (p,G)$ of the Hausdorff--Young inequality (Corollary \ref{cor:Hausdorff-Young-mG-infinite}).
Klein--Russo proved that $H (p,G)$ is determined only by the optimal constant of Young's inequality when $p' \in 2 \mathbb{Z}$ (Fact \ref{fact:Young-Hausdorff-Young-relate}).
Thus, Corollary \ref{cor:Young-Fournier} bounds $H (p,G)$ from above in this case.

There are several equivalent definitions of $H (p,G)$ (Remark \ref{rem:Hausdorff-Young-represent}).
In this paper, we adopt the definition by Russo (Definition \ref{def:Russo}) based on Kunze's argument \cite[Section 5]{MR100235}.
Let $l_\phi \colon L^2 (G) \to L^2 (G)$ be the convolution operator defined as $l_\phi ( \xi ) := \phi * \xi$ for any function $\phi \colon G \to \mathbb{C}$ on a unimodular locally compact group $G$.
For $\phi \in L^1 (G)$, $l_\phi$ is the bounded operator defined on all elements of $L^2 (G)$ by Fact \ref{fact:Young}.
We define the optimal constant $H (p,G)$ of the Hausdorff--Young inequality as follows.

\begin{definition}[Russo, {\cite[Section 1]{MR435731}}]
  \label{def:Russo}

  We define the constant $H (p,G)$ as
  \begin{align*}
    H (p,G) := \sup \{ \| \gamma \|_2^{2/p'} \mid \gamma \in L^2 (G), \; \exists \phi \in L^1 (G), \; \| \phi \|_p = 1, \; l_\gamma = | l_\phi |^{p'/2} \}
  \end{align*}
  for any unimodular locally compact group $G$ and $1 < p \leq 2$.

\end{definition}

\begin{remark}
  \label{rem:Hausdorff-Young-represent}

  There are some equivalent definitions of $H (p,G)$.

  \begin{enumerate}
    \item \label{item:H-Russo}
          Kunze defined a space $\mathcal{F} L^{p'} (G)$ of operators on $L^2 (G)$ to formulate the Hausdorff--Young inequality (Fact \ref{fact:Hausdorff-Young}) by an argument of the gage space on $L^2 (G)$ by Segal \cite[Section 5]{MR54864}.
          Russo defined $H (p,G)$ as the optimal constant of this inequality (Definition \ref{def:Russo}).
          Haagerup generalized the definition of $\mathcal{F} L^{p'} (G)$ for any von Neumann algebra \cite{MR560633} (see also \cite{MR0355628} and \cite{MR3730047}).

    \item
          When $G$ is of type-I, Lipsman formulated the Hausdorff--Young inequality by integrating the Schatten norm of the Fourier transform on the Plancherel measure \cite[Section 2]{MR425512}.
          Since the $L^2$-Schatten norm of the Fourier transform corresponds to the $L^2$-norm of the Haar measure (Plancherel's theorem), the optimal constant of this inequality is equal to that of \ref{item:H-Russo} \cite[Section 1]{MR435732}.

    \item
          Kosaki defined $L^{p'}(M)$ as the interpolation space between von Neumann algebra $M$ and its predual $M_*$, and proved that $L^{p'}(M)$ corresponds to $\mathcal{F} L^{p'} (G)$  defined by Haagerup \cite{MR735704}.
          When $M$ is the group von Neumann algebra $VN (G)$ of $G$, there is an isometric isomorphism between $M_*$ and the Fourier algebra $A (G) \subset L^\infty (G)$ \cite{MR228628}.
          Thus, $\mathcal{F} L^{p'} (G)$ is isometrically isomorphic to the interpolation space between $VN (G)$ and $A (G) \cong VN(G)_*$ \cite[Section 6]{MR2793447} and hence $H (p,G)$ can be also defined by using this normed space (see also \cite{MR677418}, \cite{MR1484866}, \cite{MR2848998}, and \cite{MR3029493}).

    \item
          Cowling--Martini--M\"{u}ller--Parcet gave $H (p,G)$ explicitly \cite[Proposition 2.1]{MR4000236}:
          \begin{align*}
            H (p,G) = \sup \{ \| | l_\phi |^{p'} \|_{L^1 (G) \to L^\infty (G)}^{1/p'} \mid \| \phi \|_p = 1 \}.
          \end{align*}

  \end{enumerate}

\end{remark}

The following fact is called the Hausdorff--Young inequality.

\begin{fact}[the Hausdorff--Young inequality, {\cite[Theorem 6 (5)]{MR100235}}]
  \label{fact:Hausdorff-Young}

  We have $H (p,G) \leq 1$ for any unimodular locally compact group $G$ and $1 < p \leq 2$.

\end{fact}

For $G = \mathbb{R} / \mathbb{Z}$, Fact \ref{fact:Hausdorff-Young} was proved by Young for particular cases of $p$ \cite[Section 6]{MR1576139} and by Hausdorff for any $1 < p \leq 2$ \cite{MR1544587}.
When $G$ is abelian, Fact \ref{fact:Hausdorff-Young} was proved by Weil \cite{MR0005741}.
There are many works on the estimate of $H (p,G)$ \cite{MR0138939} \cite{MR435731} \cite{MR385456} \cite{MR435732} \cite{MR461034} \cite{MR500308} \cite{MR499945} \cite{MR515223} \cite{MR1178028} \cite{MR1971263} \cite{MR2331133} \cite{MR3176645} \cite{MR4000236} \cite{MR4072208}.

The constant $H(p , G)$ is determined only by $\tilde{Y}_O (P,G)$ for $p' \in 2 \mathbb{Z}$ as follows.

\begin{fact}[Klein--Russo, {\cite[Lemma 2.6]{MR499945}}]
  \label{fact:Young-Hausdorff-Young-relate}

  We have
  \begin{align*}
    H (p,G) = \tilde{Y}_O ((\underbrace{p, p, \cdots , p}_{p'/2}) , G)^{2/p'}
  \end{align*}
  for any unimodular locally compact group $G$ and any $1 < p \leq 2$ with $p' \in 2 \mathbb{Z}$.

\end{fact}

The following statement holds by Corollary \ref{cor:Young-Fournier-N-general} \ref{item:Young-Fournier-N-general-obverse} and Fact \ref{fact:Young-Hausdorff-Young-relate}.

\begin{corollary}
  \label{cor:Hausdorff-Young-mG-infinite}

  We have $H (p,G) \leq B (p)$ for any unimodular locally compact group $G$ with $m (G) = \infty$ and any $1 < p \leq 2$ with $p' \in 2 \mathbb{Z}$.

\end{corollary}

When $G$ is separable and any irreducible unitary representation of $G$ is finite dimensional, Corollary \ref{cor:Hausdorff-Young-mG-infinite} was proved by Russo \cite[Theorem 3 (b)]{MR515223}.

\begin{proof}[Proof of Corollary \ref{cor:Hausdorff-Young-mG-infinite}]

  We have $H (p,G) = \tilde{Y}_O (P,G)^{2/p'}$ for $P := (p , p , \cdots p) \in \mathbb{R}^{p'/2}$ by Fact \ref{fact:Young-Hausdorff-Young-relate}.
  Since $m (G) = \infty$, we have $\tilde{Y}_O (P,G) \leq C(P)$ by Corollary \ref{cor:Young-Fournier-N-general} \ref{item:Young-Fournier-N-general-obverse}.
  Here
  \begin{align*}
    q (P) = \left( 1 - \frac{p'}{2} + \frac{p'}{2} \cdot \frac{1}{p} \right)^{-1} = 2
  \end{align*}
  holds and hence we have
  \begin{align*}
    C (P) = \frac{B (p)^{p'/2}}{B (2)} = B (p)^{p'/2}
  \end{align*}
  by $B (2) = 1$.
  Thus,
  \begin{align*}
    H (p,G)
    = \tilde{Y}_O (P,G)^{2/p'}
    \leq C(P)^{2/p'}
    = B(p)
  \end{align*}
  is obtained.
\end{proof}

\section{The necessity of the convexity of \texorpdfstring{$f$}{f}}
\label{sec:example-convex}

In this section, we see the necessity of the convexity of $f$ for $G = \mathbb{R}$.
That is, the following example implies that a function $f$ satisfying Theorem \ref{thm:arrange-convex-inequality} needs to be a convex function.
In the case of
\begin{align*}
  \phi_1 := \frac{1_{(-1,1)}}{2} , \quad
  \phi_2 := y_1 1_{(-5,-3) \cup (-1-2\lambda ,-1) \cup (1,1+2\lambda ) \cup (3,5)} + y_2 1_{(-3,-1-2\lambda ) \cup (-1,1) \cup (1+2\lambda ,3)}
\end{align*}
for $0 \leq \lambda \leq 1$ and $0 \leq y_1 \leq y_2$, we have
\begin{align*}
  \phi_1^* = \frac{1_{(-1,1)}}{2} , \quad
  \phi_2^* = y_1 1_{(-5,2\lambda -3) \cup (3-2\lambda ,5)} + y_2 1_{(2\lambda -3,3-2\lambda )}
\end{align*}
(Figure \ref{fig:phi1}, Figure \ref{fig:phi2}, and Figure \ref{fig:phi2*}).
Thus,
\begin{align*}
  \phi_1*\phi_2 (x) =
  \left\{
  \begin{aligned}
     & y_2-\frac{y_2-y_1}{2}|x|            &  & \text{if} \; |x| \leq 2\lambda                     \\
     & \lambda y_1+(1-\lambda )y_2         &  & \text{if} \; 2\lambda  \leq |x| \leq 2 \lambda + 2 \\
     & 2y_2-y_1-\frac{y_2-y_1}{2}|x|       &  & \text{if} \; 2\lambda +2 \leq |x| \leq 4           \\
     & y_1\left( 3 - \frac{|x|}{2} \right) &  & \text{if} \; 4 \leq |x| \leq 6                     \\
     & 0                                   &  & \text{if} \; 6 \leq |x|
  \end{aligned}
  \right.
\end{align*}
(Figure \ref{fig:convolution}) and
\begin{align*}
  \phi_1^* * \phi_2^* (x) =
  \left\{
  \begin{aligned}
     & y_2                                                  &  & \text{if} \; |x| \leq 2-2\lambda                  \\
     & (\lambda -1)y_1+(2-\lambda )y_2-\frac{y_2-y_1}{2}|x| &  & \text{if} \; 2-2\lambda  \leq |x| \leq 4-2\lambda \\
     & y_1                                                  &  & \text{if} \; 4-2\lambda  \leq |x| \leq 4          \\
     & y_1\left( 3 - \frac{|x|}{2} \right)                  &  & \text{if} \; 4 \leq |x| \leq 6                    \\
     & 0                                                    &  & \text{if} \; 6 \leq |x|
  \end{aligned}
  \right.
\end{align*}
(Figure \ref{fig:rearrange}) hold and hence
\begin{align*}
  \int_\mathbb{R} f\circ( \phi_1^* * \phi_2^* ) (x) dx - \int_\mathbb{R} f\circ( \phi_1*\phi_2 ) (x) dx
  = 4 (\lambda f(y_1)+(1-\lambda )f(y_2)- f(\lambda y_1+(1-\lambda )y_2) ).
\end{align*}
Since
\begin{align*}
  \lambda f(y_1)+(1-\lambda )f(y_2)- f(\lambda y_1+(1-\lambda )y_2)
  \geq 0
\end{align*}
by the convexity of $f$, we obtain \eqref{eq:arrange-convex-inequality-main}.

This example shows that a measurable function $f \colon \mathbb{R}_{\geq 0} \to \mathbb{R}$ with $f (0) = 0$ satisfies Theorem \ref{thm:arrange-convex-inequality} only if $f$ is a convex function.

\begin{figure}
  \centering
  \caption{$\phi_1 (x) = \phi_1^* (x)$}
  \label{fig:phi1}
  \begin{tikzpicture}
    \begin{axis}
      [
        width=\textwidth,
        scale only axis,
        axis x line=center,
        axis y line=center,
        xmin=-6.3,
        xmax=6.3,
        ymin=-0.5,
        ymax=3.5,
        xlabel={$x$},
        xtick={-1,1},
        ytick=\empty,
        xticklabels={$-1$,$1$},
        axis equal image=true,
      ]
      \draw [thick] (-1,1/2) -- (1,1/2);
      \draw [thick] (-6.3,0) -- (-1,0);
      \draw [thick] (1,0) -- (6.3,0);
      \draw [dotted] (-1,0) -- (-1,1/2);
      \draw [dotted] (1,0) -- (1,1/2);
      \node at (0,1/2) [above left] {$\displaystyle \frac{1}{2}$};
    \end{axis}
  \end{tikzpicture}
\end{figure}

\begin{figure}
  \centering
  \caption{$\phi_2 (x)$}
  \label{fig:phi2}
  \begin{tikzpicture}
    \begin{axis}
      [
        width=\textwidth,
        scale only axis,
        axis x line=center,
        axis y line=center,
        xmin=-6.3,
        xmax=6.3,
        ymin=-0.5,
        ymax=3.5,
        xlabel={$x$},
        xtick={-5,-3,-7/3,-1,1,7/3,3,5},
        ytick=\empty,
        xticklabels={$-5$,$-3$,\; $-1-2\lambda$,$-1$,$1$,$1+2 \lambda$,$3$,$5$},
        axis equal image=true,
      ]
      \draw [thick] (-5,1) -- (-3,1);
      \draw [thick] (-7/3,1) -- (-1,1);
      \draw [thick] (1,1) -- (7/3,1);
      \draw [thick] (3,1) -- (5,1);
      \draw [thick] (-3,3) -- (-7/3,3);
      \draw [thick] (-1,3) -- (1,3);
      \draw [thick] (7/3,3) -- (3,3);
      \draw [thick] (-6.3,0) -- (-5,0);
      \draw [thick] (5,0) -- (6.3,0);
      \draw [dotted] (-5,0) -- (-5,1) -- (5,1) -- (5,0);
      \draw [dotted] (-3,0) -- (-3,3) -- (3,3) -- (3,0);
      \draw [dotted] (-7/3,0) -- (-7/3,3);
      \draw [dotted] (-1,0) -- (-1,3);
      \draw [dotted] (1,0) -- (1,3);
      \draw [dotted] (7/3,0) -- (7/3,3);
      \node at (0,1) [above left] {$y_1$};
      \node at (0,3) [above left] {$y_2$};
    \end{axis}
  \end{tikzpicture}
\end{figure}

\begin{figure}
  \centering
  \caption{$\phi_2^* (x)$}
  \label{fig:phi2*}
  \begin{tikzpicture}
    \begin{axis}
      [
        width=\textwidth,
        scale only axis,
        axis x line=center,
        axis y line=center,
        xmin=-6.3,
        xmax=6.3,
        ymin=-0.5,
        ymax=3.5,
        xlabel={$x$},
        xtick={-5,-5/3,5/3,5},
        ytick=\empty,
        xticklabels={$-5$,$2 \lambda -3$,$3 - 2 \lambda$,$5$},
        axis equal image=true,
      ]
      \draw [thick] (-5,1) -- (-5/3,1);
      \draw [thick] (5/3,1) -- (5,1);
      \draw [thick] (-5/3,3) -- (5/3,3);
      \draw [thick] (-6.3,0) -- (-5,0);
      \draw [thick] (5,0) -- (6.3,0);
      \draw [dotted] (-5,0) -- (-5,1) -- (5,1) -- (5,0);
      \draw [dotted] (-5/3,0) -- (-5/3,3) -- (5/3,3) -- (5/3,0);
      \node at (0,1) [above left] {$y_1$};
      \node at (0,3) [above left] {$y_2$};
    \end{axis}
  \end{tikzpicture}
\end{figure}

\begin{figure}
  \centering
  \caption{$\phi_1 * \phi_2 (x)$}
  \label{fig:convolution}
  \begin{tikzpicture}
    \begin{axis}
      [
        width=\textwidth,
        scale only axis,
        axis x line=center,
        axis y line=center,
        xmin=-6.3,
        xmax=6.3,
        ymin=-0.5,
        ymax=3.5,
        xlabel={$x$},
        xtick={-6,-4,-10/3,-4/3,4/3,10/3,4,6},
        ytick=\empty,
        xticklabels={$-6$,$-4$,\; $-2 - 2 \lambda$,$-2 \lambda$,$2 \lambda$,$2 + 2 \lambda$,$4$,$6$},
        axis equal image=true,
      ]
      \draw [thick] (-6.3,0) -- (-6,0) -- (-4,1) -- (-10/3,5/3) -- (-4/3,5/3) -- (0,3) -- (4/3,5/3) -- (10/3,5/3) -- (4,1) --(6,0) -- (6.3,0);
      \draw [dotted] (-4,0) -- (-4,1) -- (4,1) -- (4,0);
      \draw [dotted] (-10/3,0) -- (-10/3,5/3) -- (10/3,5/3) -- (10/3,0);
      \draw [dotted] (-4/3,0) -- (-4/3,5/3);
      \draw [dotted] (4/3,0) -- (4/3,5/3);
      \node at (0,1) [below left] {$y_1$};
      \node at (0,5/3) [below left] {$\lambda y_1 + (1 - \lambda) y_2$};
      \node at (0,3) [above left] {$y_2$};
    \end{axis}
  \end{tikzpicture}
\end{figure}

\begin{figure}
  \centering
  \caption{$\phi_1^* * \phi_2^* (x)$}
  \label{fig:rearrange}
  \begin{tikzpicture}
    \begin{axis}
      [
        width=\textwidth,
        scale only axis,
        axis x line=center,
        axis y line=center,
        xmin=-6.3,
        xmax=6.3,
        ymin=-0.5,
        ymax=3.5,
        xlabel={$x$},
        xtick={-6,-4,-8/3,-2/3,2/3,8/3,4,6},
        ytick=\empty,
        xticklabels={$-6$,$-4$,$2 \lambda - 4$,$2 \lambda - 2$,$2 - 2 \lambda$,$4 - 2 \lambda$,$4$,$6$},
        axis equal image=true,
      ]
      \draw [thick] (-6.3,0) -- (-6,0) -- (-4,1) -- (-8/3,1) -- (-2/3,3) -- (2/3,3) -- (8/3,1) -- (4,1) --(6,0) -- (6.3,0);
      \draw [dotted] (-4,0) -- (-4,1) -- (4,1) -- (4,0);
      \draw [dotted] (-8/3,0) -- (-8/3,1);
      \draw [dotted] (8/3,0) -- (8/3,1);
      \draw [dotted] (-2/3,0) -- (-2/3,3);
      \draw [dotted] (2/3,0) -- (2/3,3);
      \node at (0,1) [above left] {$y_1$};
      \node at (0,3) [above left] {$y_2$};
    \end{axis}
  \end{tikzpicture}
\end{figure}

\section{Some properties of the rearrangement}
\label{sec:rearrange}

In this section, we show some essential properties of the rearrangement and the layer cake representation to prove Theorem \ref{thm:arrange-convex-inequality}.

\begin{lemma}
  \label{lem:rearrange-property}

  Let $\phi \colon G \to \mathbb{R}_{\geq 0}$ be a measurable function on a measure space $(G , \mu )$.

  \begin{enumerate}
    \item \label{item:rearrange-property-interval}
          We have
          \begin{align*}
            (\phi^*)^{-1} (\mathbb{R}_{> t})
            = \left( - \frac{\mu (\phi^{-1} ( \mathbb{R}_{> t}))}{2}, \frac{\mu (\phi^{-1} ( \mathbb{R}_{> t}))}{2} \right)
          \end{align*}
          for any $t \geq 0$.

    \item \label{item:rearrange-property-character}
          We have
          \begin{align*}
            1_A^{-1} ( \mathbb{R}_{> t} ) =
            \left\{
            \begin{aligned}
               & A         &  & \text{if} \; t < 1    \\
               & \emptyset &  & \text{if} \; t \geq 1
            \end{aligned}
            \right.
          \end{align*}
          for any $t \geq 0$ and subset $A \subset G$.

    \item \label{item:rearrange-property-order}
          We have
          \begin{align*}
            1_{\phi^{-1} (\mathbb{R}_{> t})}^*
            = 1_{(\phi^*)^{-1} (\mathbb{R}_{> t})}
            = 1_{( - \mu ( \phi^{-1} ( \mathbb{R}_{> t} ) ) / 2 , \mu ( \phi^{-1} ( \mathbb{R}_{> t} ) ) / 2 )}
          \end{align*}
          for any $t \geq 0$.

    \item \label{item:rearrange-property-layer}
          We have
          \begin{align*}
            \phi (g) = \int_{0}^{\infty} 1_{\phi^{-1} (\mathbb{R}_{> t})} (g) dt
          \end{align*}
          for any $g \in G$.

    \item \label{item:rearrange-property-convolution}
          If $\mu$ is a Haar measure on a unimodular locally compact group $G$, then we have
          \begin{align*}
            \phi_1 * \phi_2 (g)
            = \int_{0}^{\infty} \int_{0}^{\infty} 1_{L_1 (t_1)} * 1_{L_2 (t_2)} (g) dt_1 dt_2
          \end{align*}
          for any $g \in G$ and measurable functions $\phi_1 , \phi_2 \colon G \to \mathbb{R}_{\geq 0}$, where we write $L_1 (t) := \phi_1^{-1} (\mathbb{R}_{> t})$ and $L_2 (t) := \phi_2^{-1} (\mathbb{R}_{> t})$.

    \item \label{item:rearrange-property-increase}
          If a pointwise increasing sequence of functions $\phi_n \colon G \to \mathbb{R}_{\geq 0}$ converges pointwise to $\phi$, then $\phi_n^*$ is a pointwise increasing sequence converging pointwise to $\phi^*$.

  \end{enumerate}

\end{lemma}

\begin{proof}

  \begin{enumerate}
    \item
          Since
          \begin{align*}
            x \in (\phi^*)^{-1} (\mathbb{R}_{> t})
             & \Longleftrightarrow \phi^* (x) > t                                                                                                      \\
             & \Longleftrightarrow \inf \{ t' > 0 \mid \mu ( \phi^{-1} ( \mathbb{R}_{> t'} ) ) \leq 2 |x| \} > t                                       \\
             & \Longleftrightarrow \mu ( \phi^{-1} ( \mathbb{R}_{> t} ) ) > 2 |x|                                                                      \\
             & \Longleftrightarrow x \in \left( - \frac{\mu (\phi^{-1} ( \mathbb{R}_{> t}))}{2}, \frac{\mu (\phi^{-1} ( \mathbb{R}_{> t}))}{2} \right)
          \end{align*}
          for any $x \in \mathbb{R}$, we obtain
          \begin{align*}
            (\phi^*)^{-1} (\mathbb{R}_{> t}) = \left( - \frac{\mu (\phi^{-1} ( \mathbb{R}_{> t}))}{2}, \frac{\mu (\phi^{-1} ( \mathbb{R}_{> t}))}{2} \right).
          \end{align*}

    \item
          We have
          \begin{align*}
            1_A^{-1} ( \mathbb{R}_{> t})
            = \{ g \in G \mid 1_A (g) > t \}
            = \left\{
            \begin{aligned}
               & A         &  & \text{if} \; t < 1    \\
               & \emptyset &  & \text{if} \; t \geq 1
            \end{aligned}
            \right. .
          \end{align*}

    \item
          Since
          \begin{align*}
            1_{(\phi^*)^{-1} (\mathbb{R}_{> t})}
            = 1_{( - \mu ( \phi^{-1} ( \mathbb{R}_{> t} ) ) / 2 , \mu ( \phi^{-1} ( \mathbb{R}_{> t} ) ) / 2 )}
          \end{align*}
          by \ref{item:rearrange-property-interval}, it suffices to show
          \begin{align}
            1_{\phi^{-1} (\mathbb{R}_{> t})}^*
            = 1_{( - \mu ( \phi^{-1} ( \mathbb{R}_{> t} ) ) / 2 , \mu ( \phi^{-1} ( \mathbb{R}_{> t} ) ) / 2 )}. \label{eq:character-rearrange}
          \end{align}
          We have
          \begin{align*}
            1_{\phi^{-1} (\mathbb{R}_{> t})}^* (x)
            = \inf \{ t' > 0 \mid \mu ( 1_{\phi^{-1} (\mathbb{R}_{> t})}^{-1} ( \mathbb{R}_{> t'})) \leq 2 |x| \}
          \end{align*}
          for any $x \in \mathbb{R}$.
          Since
          \begin{align*}
            \mu ( 1_{\phi^{-1} (\mathbb{R}_{> t})}^{-1} ( \mathbb{R}_{> t'}))
            = \left\{
            \begin{aligned}
               & \mu ( \phi^{-1} (\mathbb{R}_{> t} ) ) &  & \text{if} \; t' < 1    \\
               & 0                                     &  & \text{if} \; t' \geq 1
            \end{aligned}
            \right.
          \end{align*}
          by \ref{item:rearrange-property-character}, we have
          \begin{align*}
            \inf \{ t' \geq 0 \mid \mu ( 1_{\phi^{-1} (\mathbb{R}_{> t})}^{-1} ( \mathbb{R}_{> t'})) \leq 2 |x| \}
             & = \left\{
            \begin{aligned}
               & 1 &  & \text{if} \; 2|x| < \mu ( \phi^{-1} ( \mathbb{R}_{> t}))    \\
               & 0 &  & \text{if} \; 2|x| \geq \mu ( \phi^{-1} ( \mathbb{R}_{> t}))
            \end{aligned}
            \right.                                                                                                     \\
             & = 1_{( - \mu ( \phi^{-1} ( \mathbb{R}_{> t} ) ) / 2 , \mu ( \phi^{-1} ( \mathbb{R}_{> t} ) ) / 2 )} (x).
          \end{align*}
          Thus, we obtain \eqref{eq:character-rearrange}.

    \item
          We obtain
          \begin{align*}
            \phi (g)
            = \int_{0}^{\phi (g)} dt
            = \int_{0}^{\infty} 1_{\phi^{-1} (\mathbb{R}_{> t})} (g) dt.
          \end{align*}

    \item
          We have
          \begin{align*}
            \phi_1 * \phi_2 (g)
            = \int_{G}^{} \phi_1 (g') \phi_2 (g'^{-1} g) dg'.
          \end{align*}
          Since
          \begin{align*}
            \phi_1 (g') = \int_{0}^{\infty} 1_{L_1 (t_1)} (g') dt_1, \quad
            \phi_2 (g'^{-1} g) = \int_{0}^{\infty} 1_{L_2 (t_2)} (g'^{-1} g) dt_2
          \end{align*}
          hold by \ref{item:rearrange-property-layer}, we have
          \begin{align*}
            \int_{G}^{} \phi_1 (g') \phi_2 (g'^{-1} g) dg'
            = \int_{G}^{} \int_{0}^{\infty} 1_{L_1 (t_1)} (g') dt_1 \int_{0}^{\infty} 1_{L_2 (t_2)} (g'^{-1} g) dt_2 dg'.
          \end{align*}
          We get
          \begin{align*}
            \int_{G}^{} \int_{0}^{\infty} 1_{L_1 (t_1)} (g') dt_1 \int_{0}^{\infty} 1_{L_2 (t_2)} (g'^{-1} g) dt_2 dg'
            = \int_{0}^{\infty} \int_{0}^{\infty} \int_{G}^{} 1_{L_1 (t_1)} (g') 1_{L_2 (t_2)} (g'^{-1} g) dg' dt_1 dt_2
          \end{align*}
          by Fubini's theorem.
          Since
          \begin{align*}
            \int_{G}^{} 1_{L_1 (t_1)} (g') 1_{L_2 (t_2)} (g'^{-1} g) dg'
            = 1_{L_1 (t_1)} * 1_{L_2 (t_2)} (g),
          \end{align*}
          we obtain
          \begin{align*}
            \phi_1 * \phi_2 (g)
            = \int_{0}^{\infty} \int_{0}^{\infty} \int_{G}^{} 1_{L_1 (t_1)} (g') 1_{L_2 (t_2)} (g'^{-1} g) dg' dt_1 dt_2
            = \int_{0}^{\infty} \int_{0}^{\infty} 1_{L_1 (t_1)} * 1_{L_2 (t_2)} (g) dt_1 dt_2.
          \end{align*}

    \item
          Since $\phi_n^{-1} (\mathbb{R}_{> t}) \subset \phi_{n+1}^{-1} (\mathbb{R}_{> t}) \subset \phi^{-1} (\mathbb{R}_{> t})$ for any $n$, we have $\phi_n^* \leq \phi_{n+1}^* \leq \phi^*$.

          Here $\mu ( \phi^{-1} (\mathbb{R}_{> \phi^* (x) - \epsilon})) > 2 |x|$ holds for any $x \in \mathbb{R}$ and $\epsilon > 0$.
          We fix $x \in \mathbb{R}$ and $\epsilon > 0$.
          Since the pointwise increasing sequence $\phi_n$ converges pointwise to $\phi$, there exists (sufficiently large) $n$ with $\mu ( \phi_n^{-1} (\mathbb{R}_{> \phi^* (x) - \epsilon})) > 2 |x|$.
          Thus,
          \begin{align*}
            \phi^* (x) - \epsilon \leq \phi_n^* (x) \leq \phi^* (x)
          \end{align*}
          holds.
          Therefore, we have
          \begin{align*}
            \lim_{n \to \infty} \phi_n^* (x) = \phi^* (x)
          \end{align*}
          and hence $\phi_n^*$ is a pointwise increasing sequence converging pointwise to $\phi^*$.
          \qedhere

  \end{enumerate}

\end{proof}

\begin{example}
  \label{ex:rearrange-layer}

  \begin{enumerate}
    \item \label{item:rearrange-layer-convolution}
          Let $\mu$, $G$, $\phi_1$, $\phi_2$, $L_1$, and $L_2$ be as in Lemma \ref{lem:rearrange-property} \ref{item:rearrange-property-convolution}.
          We write
          \begin{align*}
            J_1 (t) := \frac{\mu ( L_1 (t) )}{2}, \quad
            J_2 (t) := \frac{\mu ( L_2 (t) )}{2}
          \end{align*}
          for $t \geq 0$.
          Since
          \begin{align}
            1_{L_1 (t)}^*
            = 1_{( \phi_1^* )^{-1} ( \mathbb{R}_{> t})}
            = 1_{( - J_1 (t) , J_1 (t) )}, \quad
            1_{L_2 (t)}^*
            = 1_{( \phi_2^* )^{-1} ( \mathbb{R}_{> t})}
            = 1_{( - J_2 (t) , J_2 (t) )} \label{eq:phi1-phi2-layer}
          \end{align}
          by Lemma \ref{lem:rearrange-property} \ref{item:rearrange-property-order}, we have
          \begin{align}
            \phi_1^* * \phi_2^* (x)
            = \int_{0}^{\infty} \int_{0}^{\infty} 1_{(-J_1 (t_1) , J_1 (t_1))} * 1_{(-J_2 (t_2) , J_2 (t_2))} (x) dt_1 dt_2 \label{eq:rearrange-layer-convolution-represent}
          \end{align}
          for any $x \in \mathbb{R}$ by Lemma \ref{lem:rearrange-property} \ref{item:rearrange-property-convolution}.
          For any $J_1 , J_2 \geq 0$,
          \begin{align*}
            1_{(-J_1, J_1)} * 1_{(-J_2, J_2)} (x)
            = \left\{
            \begin{aligned}
               & 2 \min (J_1,J_2) &  & \text{if} \; |x| \leq |J_1 - J_2|                \\
               & J_1 + J_2 - |x|  &  & \text{if} \; |J_1 - J_2| \leq |x| \leq J_1 + J_2 \\
               & 0                &  & \text{if} \; J_1 + J_2 \leq |x|
            \end{aligned}
            \right.
          \end{align*}
          is a continuous even function and it is decreasing on $x \geq 0$.
          Thus, $\phi_1^* * \phi_2^* (x)$ is also an even function and it is decreasing on $x \geq 0$.

    \item \label{item:rearrange-layer-increase}
          In \ref{item:rearrange-layer-convolution}, there exist pointwise increasing sequences $\phi_{1,n}$ and $\phi_{2,n}$ of measurable functions such that $\phi_{1,n}$ and $\phi_{2,n}$ converge pointwise to $\phi_1$ and $\phi_2$, respectively.
          Thus, $\phi_{1,n}^*$, $\phi_{2,n}^*$, $\phi_{1,n} * \phi_{2,n}$, and $\phi_{1,n}^* * \phi_{2,n}^*$ are pointwise increasing sequences of measurable functions which converge to $\phi_1^*$, $\phi_2^*$, $\phi_1 * \phi_2$ and $\phi_1^* * \phi_2^*$, respectively, by Lemma \ref{lem:rearrange-property} \ref{item:rearrange-property-increase}.

  \end{enumerate}

\end{example}

In the case of Example \ref{ex:rearrange-layer} \ref{item:rearrange-layer-convolution}, the following lemma holds when $\phi_1^* * \phi_2^* (x)$ is finite almost everywhere.

\begin{lemma}
  \label{lem:convolution-almost-finite}

  Suppose measurable functions $\phi_1 , \phi_2 \colon G \to \mathbb{R}_{\geq 0}$ on a unimodular locally compact group $G$ satisfy that $\phi_1^* * \phi_2^* (x)$ is finite almost everywhere.

  \begin{enumerate}
    \item \label{item:convolution-almost-finite-finite}
          $\phi_1^* * \phi_2^*$ is finite (everywhere) on $\mathbb{R} \setminus \{ 0 \}$.

    \item
          $\phi_1^* * \phi_2^*$ is continuous (everywhere) on $\mathbb{R} \setminus \{ 0 \}$.

    \item
          If $\phi_1^* * \phi_2^* (0)$ is also finite, then $\phi_1^* * \phi_2^* (x)$ is continuous at $x = 0$.

  \end{enumerate}

\end{lemma}

\begin{proof}

  \begin{enumerate}
    \item
          it suffices to show $\psi (x) := \phi_1^* * \phi_2^* (x) < \infty$ for any $x \in \mathbb{R} \setminus \{ 0 \}$.
          The function $\psi$ is finite almost everywhere and hence there exists $0 < x' < |x|$ with $\psi (x') < \infty$.
          Since $\psi$ is an even function and decreasing on $\mathbb{R}_{\geq 0}$ by Example \ref{ex:rearrange-layer} \ref{item:rearrange-layer-convolution}, we obtain
          \begin{align*}
            \psi (x) \leq \psi (x') < \infty.
          \end{align*}

    \item
          Let $x , x' \in \mathbb{R} \setminus \{ 0 \}$ with $|x - x'| \leq |x|/2$.
          We denote $1_{(-J_1 (t_1) , J_1 (t_1))} * 1_{(-J_2 (t_2) , J_2 (t_2))}$ by $\alpha_{t_1,t_2}$ in Example \ref{ex:rearrange-layer} \ref{item:rearrange-layer-convolution}.
          Thus,
          \begin{align*}
            \psi (x') = \int_{0}^{\infty} \int_{0}^{\infty} \alpha_{t_1,t_2} (x') dt_1 dt_2
          \end{align*}
          holds by \eqref{eq:rearrange-layer-convolution-represent}.
          Since $\alpha_{t_1,t_2}$ is an even function and decreasing on $\mathbb{R}_{\geq 0}$ by Example \ref{ex:rearrange-layer} \ref{item:rearrange-layer-convolution}, we have
          \begin{align*}
            \alpha_{t_1,t_2} (x') \leq \alpha_{t_1,t_2} \left( \frac{x}{2} \right).
          \end{align*}
          Since $\alpha_{t_1,t_2}$ is continuous and $\psi (x/2) < \infty$ holds by \ref{item:convolution-almost-finite-finite}, we obtain
          \begin{align*}
            \lim_{x' \to x} \psi (x')
            = \lim_{x' \to x} \int_{0}^{\infty} \int_{0}^{\infty} \alpha_{t_1,t_2} (x') dt_1 dt_2
            = \int_{0}^{\infty} \int_{0}^{\infty} \alpha_{t_1,t_2} (x) dt_1 dt_2
            = \psi (x)
          \end{align*}
          by the dominated convergence theorem.
          Thus, $\psi$ is continuous at $x$.

    \item
          We have
          \begin{align*}
            \psi (0) = \int_{0}^{\infty} \int_{0}^{\infty} \alpha_{t_1,t_2} (0) dt_1 dt_2
          \end{align*}
          by \eqref{eq:rearrange-layer-convolution-represent}.
          Since $\alpha_{t_1,t_2} (x)$ is a continuous even function which is decreasing on $x \geq 0$, we have
          \begin{align*}
            \lim_{x \to 0} \psi (x)
            = \lim_{x \to 0} \int_{0}^{\infty} \int_{0}^{\infty} \alpha_{t_1,t_2} (x) dt_1 dt_2
            = \int_{0}^{\infty} \int_{0}^{\infty} \alpha_{t_1,t_2} (0) dt_1 dt_2
            = \psi (0)
          \end{align*}
          by the monotone convergence theorem and \eqref{eq:rearrange-layer-convolution-represent}.
          Thus, $\psi (x)$ is also continuous at $x = 0$.
          \qedhere

  \end{enumerate}

\end{proof}

\section{The proof of Theorem \ref{thm:arrange-convex-inequality} in the case of \texorpdfstring{$f = f_t$}{f=ft}}
\label{sec:proof-ft}

We define the convex function $f_t \colon \mathbb{R}_{\geq 0} \to \mathbb{R}_{\geq 0}$ as
\begin{align*}
  f_t (y) :=
  \left\{
  \begin{aligned}
     & y - t &  & \text{if} \; y \geq t \\
     & 0     &  & \text{if} \; y \leq t
  \end{aligned}
  \right.
\end{align*}
for $t \geq 0$.
In this section, we show Theorem \ref{thm:arrange-convex-inequality} for $f = f_t$.
That is, we show
\begin{align}
  \| f_t \circ ( \phi_1 * \phi_2) \| \leq \| f_t \circ ( \phi_1^* * \phi_2^* ) \| \label{eq:arrange-convex-ft}
\end{align}
for any $t \geq 0$, where we denote the $L^1$-norm $\| \cdot \|_1$ by simply $\| \cdot \|$.
In Subsection \ref{subsec:proof-ft-lemma}, we will prepare a lemma (Lemma \ref{lem:ft-integral}) to prove Theorem \ref{thm:arrange-convex-inequality}.
In Subsection \ref{subsec:proof-ft-proof}, we will show Theorem \ref{thm:arrange-convex-inequality} by using this lemma.

\subsection{The lemma of Theorem \ref{thm:arrange-convex-inequality}}
\label{subsec:proof-ft-lemma}

In this subsection, we show the following lemma to prove Theorem \ref{thm:arrange-convex-inequality}.

\begin{lemma}
  \label{lem:ft-integral}

  We have
  \begin{align*}
    f_t \left( \int_{\Omega}^{} \alpha (\omega) d \omega \right)
    \leq \int_{\Omega}^{} f_{T (\omega)} \circ \alpha (\omega) d \omega, \quad
    t := \int_{\Omega}^{} T (\omega) d \omega
  \end{align*}
  for any measurable functions $T , \alpha \colon \Omega \to \mathbb{R}$ on a measure space $\Omega$.
  Furthermore, the equality holds if $T(\omega) \leq \alpha (\omega)$ almost everywhere.

\end{lemma}

\begin{proof}

  When $\int_{\Omega}^{} \alpha (\omega) d \omega \leq t$, we have
  \begin{align*}
    f_t \left( \int_{\Omega}^{} \alpha (\omega) d \omega \right)
    = 0
    \leq \int_{\Omega}^{} f_{T (\omega)} \circ \alpha (\omega) d \omega.
  \end{align*}
  When $\int_{\Omega}^{} \alpha (\omega) d \omega \geq t$, we have
  \begin{align*}
    f_t \left( \int_{\Omega}^{} \alpha (\omega) d \omega \right)
    = \int_{\Omega}^{} \alpha (\omega) d \omega - t
    = \int_{\Omega}^{} ( \alpha (\omega) - T ( \omega ) ) d \omega
    \leq \int_{\Omega}^{} f_{T (\omega)} \circ \alpha (\omega) d \omega.
  \end{align*}
  Furthermore, the equality holds if $T(\omega) \leq \alpha (\omega)$ almost everywhere.
\end{proof}

\begin{example}
  \label{ex:ft-integral-apply}

  Let $\mu$, $G$, $\phi_1$, $\phi_2$, $L_1$, $L_2$, $J_1 (t)$, and $J_2 (t)$ be as in Example \ref{ex:rearrange-layer} \ref{item:rearrange-layer-convolution}.
  When we fix $x_0 \geq 0$, we have
  \begin{align*}
    t
    := \phi_1^* * \phi_2^* (x_0)
    = \int_{0}^{\infty} \int_{0}^{\infty} T (t_1,t_2) dt_1 dt_2
  \end{align*}
  by \eqref{eq:rearrange-layer-convolution-represent}, where $T (t_1,t_2) := 1_{(-J_1 (t_1), J_1 (t_1))} * 1_{(-J_2 (t_2), J_2 (t_2))} (x_0)$.
  We apply Lemma \ref{lem:ft-integral} to $\Omega = \mathbb{R}_{\geq 0} \times \mathbb{R}_{\geq 0}$.

  \begin{enumerate}
    \item
          We fix $g \in G$ and let $\alpha (t_1,t_2) := 1_{L_1 (t_1)} * 1_{L_2 (t_2)} (g)$.
          \begin{align*}
            \int_{0}^{\infty} \int_{0}^{\infty} \alpha (t_1,t_2) dt_1 dt_2 = \phi_1 * \phi_2 (g)
          \end{align*}
          holds by Lemma \ref{lem:rearrange-property} \ref{item:rearrange-property-convolution}.
          Thus, we have
          \begin{align}
            f_t \circ (\phi_1 * \phi_2) (g)
            \leq \int_{0}^{\infty} \int_{0}^{\infty} f_{T(t_1,t_2)} \circ ( 1_{L_1 (t_1)} * 1_{L_2 (t_2)} ) (g) dt_1 dt_2 \label{eq:ft-integral-apply-convolution}
          \end{align}
          by Lemma \ref{lem:ft-integral}.

    \item
          We fix $x \in \mathbb{R}$ with $|x| \leq x_0$ and let $\alpha (t_1,t_2) := 1_{(-J_1(t_1),J_1(t_1))} * 1_{(-J_2(t_2),J_2(t_2))} (x)$.
          By \eqref{eq:rearrange-layer-convolution-represent}, $\phi_1^* * \phi_2^* (x)$ is given as
          \begin{align}
            \phi_1^* * \phi_2^* (x)
            = \int_{0}^{\infty} \int_{0}^{\infty} \alpha (t_1,t_2) dt_1 dt_2. \label{eq:convolution-layer-alpha}
          \end{align}
          Since $1_{(-J_1(t_1),J_1(t_1))} * 1_{(-J_2(t_2),J_2(t_2))}$ is an even function and decreasing on $\mathbb{R}_{\geq 0}$, we have $T (t_1,t_2) \leq \alpha (t_1,t_2)$ by $|x| \leq x_0$.
          Thus,
          \begin{align*}
            f_t \circ \left( \int_{0}^{\infty} \int_{0}^{\infty} \alpha (t_1,t_2) dt_1 dt_2 \right)
            = \int_{0}^{\infty} \int_{0}^{\infty} f_{T(t_1,t_2)} \circ ( \alpha ( t_1 , t_2 ) ) (x) dt_1 dt_2
          \end{align*}
          holds by applying Lemma \ref{lem:ft-integral} and hence we have
          \begin{align}
            f_t \circ ( \phi_1^* * \phi_2^* ) (x)
            = \int_{0}^{\infty} \int_{0}^{\infty} f_{T(t_1,t_2)} \circ ( 1_{(-J_1(t_1),J_1(t_1))} * 1_{(-J_2(t_2),J_2(t_2))} ) (x) dt_1 dt_2 \label{eq:ft-integral-apply-rearrange}
          \end{align}
          by \eqref{eq:convolution-layer-alpha}.

  \end{enumerate}

\end{example}

\subsection{The proof of Theorem \ref{thm:arrange-convex-inequality}}
\label{subsec:proof-ft-proof}

In this subsection, we show Theorem \ref{thm:arrange-convex-inequality} for $f = f_t$ in three steps.

\begin{enumerate}
  \item \label{item:ft-proof-character-finite}
        In the case where $\phi_1$ and $\phi_2$ are the characteristic functions of any measurable sets of finite volume:

        It suffices to show \eqref{eq:arrange-convex-ft} when $\| \phi_1 \| \leq \| \phi_2 \|$.
        If $\| \phi_1 \| \leq t$, then we have $\phi_1 * \phi_2 \leq t$ by $\phi_2 \leq 1$ and hence we obtain
        \begin{align*}
          \| f_t \circ ( \phi_1 * \phi_2 ) \| = 0 \leq \| f_t \circ ( \phi_1^* * \phi_2^* ) \|.
        \end{align*}
        If $t \leq \| \phi_1 \|$, then
        \begin{align*}
          \| f_t \circ ( \phi_1 * \phi_2 ) \|
          \leq ( \| \phi_1 \| - t ) ( \| \phi_2 \| - t )
          = \| f_t \circ ( \phi_1^* * \phi_2^* ) \|
        \end{align*}
        holds by \eqref{eq:supp-small} \cite[Theorem 1.1]{satomi2021inequality}.

  \item \label{item:ft-proof-character-general}
        In the case where $\phi_1$ and $\phi_2$ are the characteristic functions of any measurable sets (which are not necessary of finite volume):

        There exist increasing sequences $A_1^{(n)}, A_2^{(n)} \subset G$ of measurable sets of finite volume such that $1_{A_1^{(n)}}$ and $1_{A_2^{(n)}}$ converge pointwise to $\phi_1$ and $\phi_2$, respectively.
        Thus, we have
        \begin{align*}
          \lim_{n \to \infty} \left\| f_t \circ \left( 1_{A_1^{(n)}} * 1_{A_2^{(n)}} \right) \right\|
          = \| f_t \circ ( \phi_1 * \phi_2 ) \| , \quad
          \lim_{n \to \infty} \left\| f_t \circ \left( 1_{A_1^{(n)}}^* * 1_{A_2^{(n)}}^* \right) \right\|
          = \| f_t \circ ( \phi_1^* * \phi_2^* ) \|
        \end{align*}
        by Example \ref{ex:rearrange-layer} \ref{item:rearrange-layer-increase} and the monotone convergence theorem.
        In addition,
        \begin{align*}
          \left\| f_t \circ \left( 1_{A_1^{(n)}} * 1_{A_2^{(n)}} \right) \right\|
          \leq \left\| f_t \circ \left( 1_{A_1^{(n)}}^* * 1_{A_2^{(n)}}^* \right) \right\|
        \end{align*}
        holds for any $n$ by \ref{item:ft-proof-character-finite} and hence we obtain
        \begin{align*}
          \| f_t \circ ( \phi_1 * \phi_2 ) \|
          = \lim_{n \to \infty} \left\| f_t \circ \left( 1_{A_1^{(n)}} * 1_{A_2^{(n)}} \right) \right\|
          \leq \lim_{n \to \infty} \left\| f_t \circ \left( 1_{A_1^{(n)}}^* * 1_{A_2^{(n)}}^* \right) \right\|
          = \| f_t \circ ( \phi_1^* * \phi_2^* ) \|.
        \end{align*}

  \item
        In the general case:

        When
        \begin{align}
          \lim_{x \to \infty} \phi_1^* * \phi_2^* (x) < t \leq \phi_1^* * \phi_2^* (0) \label{eq:t-reduce}
        \end{align}
        does not hold, it can be reduced to the case of \eqref{eq:t-reduce} (by using the monotone convergence theorem if necessary) or we have $\| f_t \circ (\phi_1^* * \phi_2^* ) \| = \infty$.
        In the case of \eqref{eq:t-reduce}, there exists $x_0 \geq 0$ with $t = \phi_1^* * \phi_2^* (x_0)$ by Lemma \ref{lem:convolution-almost-finite} and the intermediate value theorem.
        Let $T (t_1,t_2)$ be as in Example \ref{ex:ft-integral-apply}.
        Then
        \begin{align*}
          \| f_t \circ (\phi_1 * \phi_2) \|
          \leq \int_{0}^{\infty} \int_{0}^{\infty} \| f_{T(t_1,t_2)} \circ ( 1_{L_1 (t_1)} * 1_{L_2 (t_2)} ) \| dt_1 dt_2
        \end{align*}
        holds by integrating \eqref{eq:ft-integral-apply-convolution} over $g \in G$.
        We have
        \begin{align*}
          \| f_{T(t_1,t_2)} \circ ( 1_{L_1 (t_1)} * 1_{L_2 (t_2)} ) \|
          \leq \| f_{T(t_1,t_2)} \circ ( 1_{L_1 (t_1)}^* * 1_{L_2 (t_2)}^* ) \|
        \end{align*}
        by \ref{item:ft-proof-character-general} and hence
        \begin{align*}
          \| f_{T(t_1,t_2)} \circ ( 1_{L_1 (t_1)} * 1_{L_2 (t_2)} ) \|
          \leq \| f_{T(t_1,t_2)} \circ ( 1_{(-J_1(t_1),J_1(t_1))} * 1_{(-J_2(t_2),J_2(t_2))} ) \|
        \end{align*}
        by \eqref{eq:phi1-phi2-layer}.
        Thus, we obtain
        \begin{align*}
          \| f_t \circ (\phi_1 * \phi_2) \|
          \leq \int_{0}^{\infty} \int_{0}^{\infty} \| f_{T(t_1,t_2)} \circ ( 1_{(-J_1(t_1),J_1(t_1))} * 1_{(-J_2(t_2),J_2(t_2))} ) \| dt_1 dt_2
          = \| f_t \circ (\phi_1^* * \phi_2^* ) \|
        \end{align*}
        by \eqref{eq:ft-integral-apply-rearrange}.

\end{enumerate}

Therefore, we have \eqref{eq:arrange-convex-ft} and hence Theorem \ref{thm:arrange-convex-inequality} is proved for $f = f_t$.

\section{The proof of Theorem \ref{thm:arrange-convex-inequality} when \texorpdfstring{$\phi_1 * \phi_2$}{phi1*phi2} is integrable}
\label{sec:integrable-proof}

In this section, we show Theorem \ref{thm:arrange-convex-inequality} by an argument of Wang--Madiman when $\phi_1 * \phi_2$ is integrable.

\begin{fact}[{\cite[Section VII]{MR3252379}}]
  \label{fact:convex-ft-approximate}

  Suppose a convex function $f \colon \mathbb{R}_{\geq 0} \to \mathbb{R}$ with $f (0)=0$ is continuous at $0$ and satisfies $f_+' (0) > - \infty$, where $f_+'$ denotes the right derivative.
  Then there exists a Borel measure $\nu$ on $\mathbb{R}_{>0}$ such that $\nu (( t_1 , t_2 ]) = f_+' (t_2) - f_+' (t_1)$ for any $0 \leq t_1 \leq t_2$.
  Furthermore, we have
  \begin{align*}
    f(y) = f_+' (0) y  + \int_{0}^{\infty} f_t (y) d \nu (t)
  \end{align*}
  for any $y \geq 0$.

\end{fact}

\begin{example}
  \label{ex:convex-decompose}

  Suppose $G$, $\phi_1$, and $\phi_2$ are as in Theorem \ref{thm:arrange-convex-inequality} and $\phi_1 * \phi_2$ is integrable.
  Since
  \begin{align*}
    f \circ ( \phi_1 * \phi_2 ) (g)
    = f_+' (0) \phi_1 * \phi_2 (g) + \int_{0}^{\infty} f_t \circ (\phi_1 * \phi_2) (g) d\nu (t)
  \end{align*}
  by Fact \ref{fact:convex-ft-approximate}, we have
  \begin{align}
    \int_{G}^{} f \circ ( \phi_1 * \phi_2 ) (g) dg
    = f_+' (0) \| \phi_1 * \phi_2 \| + \int_{0}^{\infty} \| f_t \circ (\phi_1 * \phi_2) \| d\nu (t) \label{eq:convex-decompose-equation}
  \end{align}
  by integrating over $g \in G$.
  Similarly, we have
  \begin{align}
    \int_{\mathbb{R}}^{} f \circ ( \phi_1^* * \phi_2^* ) (x) dx
    = f_+' (0) \| \phi_1^* * \phi_2^* \| + \int_{0}^{\infty} \| f_t \circ (\phi_1^* * \phi_2^*) \| d\nu (t). \label{eq:convex-decompose-rearrange}
  \end{align}
  Since $\phi_1 * \phi_2$ and $\phi_1^* * \phi_2^*$ are integrable by Fubini's theorem and
  \begin{align}
    \| \phi_1 * \phi_2 \| =  \| \phi_1^* * \phi_2^* \|, \label{eq:Fubini}
  \end{align}
  \eqref{eq:convex-decompose-equation} and \eqref{eq:convex-decompose-rearrange} are well-defined.
  Thus, we obtain \eqref{eq:arrange-convex-inequality-main} by \eqref{eq:arrange-convex-ft} and \eqref{eq:Fubini}.

\end{example}

Theorem \ref{thm:arrange-convex-inequality} follows from Example \ref{ex:convex-decompose} when $\phi_1 * \phi_2$ is integrable, $f(y)$ is continuous at $y = 0$, and $f_+' (0) > - \infty$.
Here we show Theorem \ref{thm:arrange-convex-inequality} by using Example \ref{ex:convex-decompose} when $f (y)$ is not continuous at $y = 0$ or $f_+' (0) = - \infty$.
We define $f_{(n)} \colon \mathbb{R}_{\geq 0} \to \mathbb{R}$ as
\begin{align*}
  f_{(n)} (y)
  := \left\{
  \begin{aligned}
     & n f \left( \frac{1}{n} \right) y &  & \text{if} \; y \leq \frac{1}{n} \\
     & f (y)                            &  & \text{if} \; y \geq \frac{1}{n}
  \end{aligned}
  \right. .
\end{align*}
The function $f_{(n)} (y)$ is convex and continuous at $y = 0$.
We have
\begin{align}
  f_{(n)+}' (0) = n f \left( \frac{1}{n} \right) > - \infty . \label{eq:fn-right-derivative}
\end{align}
Since $f$ is not continuous at $y = 0$ or $f_+' (0) = - \infty$, there exists $n' \in \mathbb{Z}_{\geq 1}$ such that $f$ is negative on $(0,1/n']$.
We have $0 \geq f_{(n)} \geq f$ on $[0,1/n']$ for any $n \geq n'$ by the convexity of $f$.

\begin{remark}
  \label{rem:rearrange-convex-remark-finite}

  Since Theorem \ref{thm:arrange-convex-inequality} is clear when $\int_{G}^{} f \circ ( \phi_1 * \phi_2 ) (g) dg = - \infty$ or $\int_{\mathbb{R}}^{} f \circ ( \phi_1^* * \phi_2^* ) (x) dx = \infty$, it suffices to show Theorem \ref{thm:arrange-convex-inequality} when
  \begin{align}
    \int_{G}^{} f \circ ( \phi_1 * \phi_2 ) (g) dg              & > - \infty, \label{eq:rearrange-convex-remark-finite-negative} \\
    \int_{\mathbb{R}}^{} f \circ ( \phi_1^* * \phi_2^* ) (x) dx & < \infty. \label{eq:rearrange-convex-remark-finite-positive}
  \end{align}

\end{remark}

For any $n \geq n'$, we have $0 \geq f_{(n)} \geq f$ on $[0,1/n']$ and hence
\begin{align}
  \int_{G}^{} f_{(n)} \circ ( \phi_1 * \phi_2 ) (g) dg
  \leq \int_{\mathbb{R}}^{} f_{(n)} \circ ( \phi_1^* * \phi_2^* ) (x) dx
  < \infty \label{eq:fn-compare-finite}
\end{align}
by Example \ref{ex:convex-decompose}, \eqref{eq:fn-right-derivative} and \eqref{eq:rearrange-convex-remark-finite-positive}.
Since the pointwise decreasing sequence $f_{(n)}$ converges to $f$, we have
\begin{align}
  \lim_{n \to \infty} \int_{G}^{} f_{(n)} \circ ( \phi_1 * \phi_2 ) (g) dg
   & = \int_{G}^{} f \circ ( \phi_1 * \phi_2 ) (g) dg, \label{eq:phi1-phi2-converge}               \\
  \lim_{n \to \infty} \int_{\mathbb{R}}^{} f_{(n)} \circ ( \phi_1^* * \phi_2^* ) (x) dx
   & = \int_{\mathbb{R}}^{} f \circ ( \phi_1^* * \phi_2^* ) (x) dx \label{eq:phi1*-phi2*-converge}
\end{align}
by the monotone convergence theorem.
Thus, we have \eqref{eq:arrange-convex-inequality-main} by \eqref{eq:fn-compare-finite}, \eqref{eq:phi1-phi2-converge} and \eqref{eq:phi1*-phi2*-converge} and hence we obtain Theorem \ref{thm:arrange-convex-inequality} when $\phi_1 * \phi_2$ is integrable.

\section{The proof of Theorem \ref{thm:arrange-convex-inequality} in the general case}
\label{sec:rearrange-proof-general}

In this section, we show Theorem \ref{thm:arrange-convex-inequality} when $\phi_1 * \phi_2$ may not be integrable.
We will show Theorem \ref{thm:arrange-convex-inequality} when $f \circ ( \phi_1^* * \phi_2^*)$ is positive in Subsection \ref{subsec:proof-positive}, and when $f$ is monotonically decreasing in Subsection \ref{subsec:decrease-proof}.
In Subsection \ref{subsec:rearrange-proof-general-complete}, we will complete the proof of Theorem \ref{thm:arrange-convex-inequality} by showing that the cases of Section \ref{sec:integrable-proof}, Subsection \ref{subsec:proof-positive}, and Subsection \ref{subsec:decrease-proof} exhaust all the cases of Theorem \ref{thm:arrange-convex-inequality}.

\subsection{The case where \texorpdfstring{$f \circ ( \phi_1^* * \phi_2^*)$}{f◦(φ1**φ2*)} is positive}
\label{subsec:proof-positive}

When $f \circ ( \phi_1^* * \phi_2^*)$ is positive, Theorem \ref{thm:arrange-convex-inequality} can be obtained by a similar argument as in Section \ref{sec:integrable-proof}.
In this case, we may assume that $f$ is positive by replacing $f$ by $\max (f,0)$, and \eqref{eq:rearrange-convex-remark-finite-positive} holds by Remark \ref{rem:rearrange-convex-remark-finite}.
Since $f (0) = 0$ and $f$ is positive, $f (y)$ is continuous at $y = 0$ and $f_{+}' (0) \geq 0$ holds.
Thus, there exists a Borel measure $\nu$ in Fact \ref{fact:convex-ft-approximate}.
Since \eqref{eq:convex-decompose-rearrange} holds by a similar argument as in Example \ref{ex:convex-decompose} (both sides of \eqref{eq:convex-decompose-rearrange} are well-defined by \eqref{eq:rearrange-convex-remark-finite-positive}), both sides of \eqref{eq:convex-decompose-equation} are well-defined and \eqref{eq:convex-decompose-equation} holds by \eqref{eq:arrange-convex-ft} and the dominated convergence theorem.
Thus, we have \eqref{eq:arrange-convex-inequality-main} by \eqref{eq:arrange-convex-ft} and hence we obtain Theorem \ref{thm:arrange-convex-inequality} when $f \circ ( \phi_1^* * \phi_2^*)$ is positive.

\subsection{The case where \texorpdfstring{$f$}{f} is monotonically decreasing}
\label{subsec:decrease-proof}

In this subsection, we show Theorem \ref{thm:arrange-convex-inequality} when $f$ is monotonically decreasing.
There exist pointwise increasing sequences $\phi_{1,n}, \phi_{2,n} \colon G \to \mathbb{R}$ of integrable functions converging to $\phi_1$ and $\phi_2$, respectively.
By Example \ref{ex:rearrange-layer} \ref{item:rearrange-layer-increase}, $\phi_{1,n} * \phi_{2,n}$ and $\phi_{1,n}^* * \phi_{2,n}^*$ converge pointwise to $\phi_1 * \phi_2$ and $\phi_1^* * \phi_2^*$, respectively.
Since $f$ is monotonically decreasing, $f \circ ( \phi_{1,n} * \phi_{2,n} )$ and $f \circ ( \phi_{1,n}^* * \phi_{2,n}^* )$ are pointwise decreasing sequences.
The convex function $f$ is negative by $f(0)=0$ and hence
\begin{align*}
  \int_{G}^{} f \circ ( \phi_1 * \phi_2 ) (g) dg
   & = \lim_{n \to \infty} \int_{G}^{} f \circ ( \phi_{1,n} * \phi_{2,n} ) (g) dg,             \\
  \int_{\mathbb{R}}^{} f \circ ( \phi_1^* * \phi_2^* ) (x) dx
   & = \lim_{n \to \infty} \int_{\mathbb{R}}^{} f \circ ( \phi_{1,n}^* * \phi_{2,n}^* ) (x) dx
\end{align*}
by the monotone convergence theorem.
Since $\phi_{1,n} * \phi_{2,n}$ is integrable for any $n$, we have
\begin{align*}
  \int_{G}^{} f \circ ( \phi_{1,n} * \phi_{2,n} ) (g) dg
  \leq \int_{\mathbb{R}}^{} f \circ ( \phi_{1,n}^* * \phi_{2,n}^* ) (x) dx
\end{align*}
by Section \ref{sec:integrable-proof}.
Thus, \eqref{eq:arrange-convex-inequality-main} holds and hence we obtain Theorem \ref{thm:arrange-convex-inequality} when $f$ is monotonically decreasing.

\subsection{The completion of Theorem \ref{thm:arrange-convex-inequality}}
\label{subsec:rearrange-proof-general-complete}

Theorem \ref{thm:arrange-convex-inequality} was obtained in some cases in Section \ref{sec:integrable-proof}, Subsection \ref{subsec:proof-positive}, and Subsection \ref{subsec:decrease-proof}.
In this subsection, we complete the proof of Theorem \ref{thm:arrange-convex-inequality} by showing that these cases exhaust all the cases of Theorem \ref{thm:arrange-convex-inequality}.
That is, we show the following lemma.

\begin{lemma}
  \label{lem:convex-convolution-all-case}

  Suppose a convex function $f \colon \mathbb{R}_{\geq 0} \to \mathbb{R}$ with $f(0) = 0$ is not monotonically decreasing.
  We let $t > 0$ with $f (t) < 0$.

  \begin{enumerate}
    \item \label{item:convex-convolution-all-case-zero}
          There exists $t' > t$ with $f(t') = 0$.
          Furthermore,
          \begin{align}
            f (y) \leq 0 \label{eq:convex-convolution-all-case-zero-negative}
          \end{align}
          holds for any $0 \leq y \leq t'$, and
          \begin{align}
            f (y) \geq 0 \label{eq:convex-convolution-all-case-zero-positive}
          \end{align}
          holds for any $y \geq t'$.

    \item \label{item:convex-convolution-all-case-bound}
          Let $t'$ be as in \ref{item:convex-convolution-all-case-zero}.
          Then we have
          \begin{align*}
            y \leq \frac{t \max ( - f ( y ) , 0 )}{- f ( t )} + \frac{t' f_t ( y )}{t' - t}
          \end{align*}
          for any $y \geq 0$.

    \item \label{item:convex-convolution-all-case-large-part}
          Suppose measurable functions $\phi_1, \phi_2 \colon G \to \mathbb{R}_{\geq 0}$ on a unimodular locally compact group $G$ satisfy \eqref{eq:arrange-convex-inequality-finite} and \eqref{eq:rearrange-convex-remark-finite-positive}.
          Then we have
          \begin{align*}
            \| f_{\phi_1^* * \phi_2^* (x_0)} \circ ( \phi_1 * \phi_2 ) \|
            \leq \| f_{\phi_1^* * \phi_2^* (x_0)} \circ ( \phi_1^* * \phi_2^* ) \|
            < \infty
          \end{align*}
          for any $x_0 \in \mathbb{R}$ with $f \circ ( \phi_1^* * \phi_2^* ) (x_0) < 0$.

    \item
          Suppose there exists $x_0 \in \mathbb{R}$ in \ref{item:convex-convolution-all-case-large-part}.
          If \eqref{eq:rearrange-convex-remark-finite-negative} holds, then $\phi_1 * \phi_2$ is integrable.

  \end{enumerate}

\end{lemma}

\begin{proof}

  \begin{enumerate}
    \item
          Since the convex function $f$ is not monotonically decreasing, there exists $t'' > t$ with $f(t'') > 0$.
          The convex function $f$ is continuous on $\mathbb{R}_{> 0}$ and hence there exists $t' > t$ with $f (t') = 0$ by the intermediate value theorem.
          Since $f (0) = f (t') = 0$ holds and $f$ is convex, we obtain $f (y) \leq 0$ for any $0 \leq y \leq t'$, and $f (y) \geq 0$ for any $y \geq t'$.

    \item
          When $0 \leq y \leq t$, we have \eqref{eq:convex-convolution-all-case-zero-negative} and hence
          \begin{align*}
            y
            \leq \frac{t f(y)}{f (t)}
            = \frac{t \max ( - f ( y ) , 0 )}{- f ( t )}
            \leq \frac{t \max ( - f ( y ) , 0 )}{- f ( t )} + \frac{t' f_t ( y )}{t' - t}
          \end{align*}
          by the convexity of $f$.

          When $t \leq y \leq t'$, we have
          \begin{align}
            y
            = \frac{t (t' - y) + t' (y - t)}{t' - t}
            = \frac{t (t' - y) + t' f_t (y)}{t' - t}. \label{eq:y-decompose}
          \end{align}
          Since $f$ is convex,
          \begin{align*}
            \frac{t' - y}{t' - t}
            \leq \frac{f(y)}{f(t)}
          \end{align*}
          holds by \eqref{eq:convex-convolution-all-case-zero-negative}.
          Thus, we obtain
          \begin{align*}
            y
            = \frac{t (t' - y) + t' f_t (y)}{t' - t}
            \leq \frac{t \max ( - f ( y ) , 0 )}{- f ( t )} + \frac{t' f_t ( y )}{t' - t}.
          \end{align*}
          When $ t' \leq y$, we also have \eqref{eq:y-decompose} and hence we obtain
          \begin{align*}
            y
            \leq \frac{t' f_t (y)}{t' - t}
            \leq \frac{t \max ( - f ( y ) , 0 )}{- f ( t )} + \frac{t' f_t ( y )}{t' - t}.
          \end{align*}

    \item
          Let $\beta := \phi_1 * \phi_2$, $\psi := \phi_1^* * \phi_2^*$ and $t := \psi (x_0)$.
          Since $\psi$ is an even function by Example \ref{ex:rearrange-layer} \ref{item:rearrange-layer-convolution}, we may assume $x_0 \geq 0$.
          We have $\| f_t \circ \beta \| \leq \| f_t \circ \psi \|$ by \eqref{eq:arrange-convex-ft} and hence it suffices to show $\| f_t \circ \psi \| < \infty$.
          Let $t'$ be as in \ref{item:convex-convolution-all-case-zero}.
          We have
          \begin{align*}
            f_{t'} (y) \leq \frac{(t'-t) f(y)}{- f(t)}
          \end{align*}
          for any $y \geq t'$ by the convexity of $f$.
          Thus,
          \begin{align*}
            \| f_{t'} \circ \psi \| \leq \frac{t'-t}{- f(t)} \int_{G}^{} f \circ \psi (g) dg
          \end{align*}
          holds and hence we have $\| f_{t'} \circ \psi \| < \infty$ by \eqref{eq:rearrange-convex-remark-finite-positive}.
          Since $\psi$ is an even function which is monotonically decreasing on $\mathbb{R}_{\geq 0}$ by Example \ref{ex:rearrange-layer} \ref{item:rearrange-layer-convolution}, we obtain
          \begin{align*}
            \| f_t \circ \psi \|
            = \int_{-x_0}^{x_0} ( \psi (x) - t ) dx
            \leq \| f_{t'} \circ \psi \| + ( t' - t ) \int_{-x_0}^{x_0} dx
            < \infty .
          \end{align*}

    \item
          Let $\beta$, $t$, and $t'$ be as in \ref{item:convex-convolution-all-case-large-part}.
          Since
          \begin{align*}
            \| \beta \|
            \leq \frac{t \| \max (- f \circ \beta , 0) \|}{- f(t)} + \frac{t' \| f_t \circ \beta \|}{t' - t}
          \end{align*}
          by \ref{item:convex-convolution-all-case-bound} and \eqref{eq:convex-convolution-all-case-zero-positive}, we obtain $\| \beta \| < \infty$ by \ref{item:convex-convolution-all-case-large-part} and \eqref{eq:rearrange-convex-remark-finite-negative}.
          \qedhere

  \end{enumerate}

\end{proof}

By Lemma \ref{lem:convex-convolution-all-case}, the cases of Section \ref{sec:integrable-proof}, Subsection \ref{subsec:proof-positive}, and Subsection \ref{subsec:decrease-proof} exhaust all the cases of Theorem \ref{thm:arrange-convex-inequality}.
Thus, we complete the proof of Theorem \ref{thm:arrange-convex-inequality}.

\section*{Acknowledgement}

This work was supported by JSPS KAKENHI Grant Number JP19J22628 and Leading Graduate Course for Frontiers of Mathematical Sciences and Physics (FMSP).
The author would like to thank his advisor Toshiyuki Kobayashi for his support.
The author is also grateful to Yuichiro Tanaka, Toshihisa Kubo and the anonymous referees for their careful comments.

\printbibliography

\noindent
Takashi Satomi: Graduate School of Mathematical Sciences, The University of Tokyo, 3-8-1 Komaba Meguro-ku Tokyo 153-8914, Japan.

\noindent
E-mail: tsatomi@ms.u-tokyo.ac.jp

\end{document}